\documentclass[11pt,a4paper]{amsart}
\usepackage{amscd,amssymb,amsmath,latexsym}  
\usepackage{graphicx,epsfig,color,url}  
\usepackage[all]{xy}  
  
\def\ov#1{{\overline{#1}}}

\def\wt#1{{\widetilde{#1}}}  
  
\newcommand{\MV}{\operatorname{MV}}  
\newcommand{\Conv}{\operatorname{Conv}}  
\newcommand{\newton}{\operatorname{N}}  
  
\newcommand{\ord}{\operatorname{ord}}  
  
\newcommand{\num}{{\operatorname{num}}}  
\newcommand{\den}{{\operatorname{den}}}  
\newcommand{\h}{{\operatorname{h}}}  
\newcommand{\Card}{\operatorname{Card}}  
  
\newcommand{\Id}{\operatorname{Id}}  
\newcommand{\Char}{\operatorname{char}}  
\newcommand{\MI}{\operatorname{MI}}  
\newcommand{\sign}{\operatorname{sign}}  
  
\newcommand{\Res}{\operatorname{Res}}  
\newcommand{\init}{\operatorname{init}}

\newcommand{\Div}{\operatorname{Div}}  
\newcommand{\End}{\operatorname{End}}  
\newcommand{\Aut}{\operatorname{Aut}}  
  
\def \F{\mathbb{F}}  
  
\def \N{\mathbb{N}}  
\def \P{\mathbb{P}}  
  
\def \R{\mathbb{R}}  
\def \T{\mathbb{T}}  
\def \Z{\mathbb{Z}}  
\def \K{\mathbb{K}}  
\def \A{\mathbb{A}}

\def\tJ {{L}}

\def\cP {{\mathcal P}}

\def\cT {{\mathcal T}}

\newcommand{\bfu}{{\boldsymbol{u}}}

\numberwithin{equation}{section}  
  
\theoremstyle{definition}  
\newtheorem{definition}{Definition}[section]  
  
\newtheorem{remark}[definition]{Remark}  
\newtheorem{example}[definition]{Example}

\theoremstyle{plain}  
\newtheorem{lemma}[definition]{Lemma}  
\newtheorem{proposition}[definition]{Proposition}  
\newtheorem{theorem}[definition]{Theorem}  
\newtheorem{corollary}[definition]{Corollary}

\newenvironment{Proof}[1]{{\it #1}}{\hfill\mbox{$\Box$}}  
  
\begin{document}  
\title{The Newton polygon of a rational plane curve}  
  
%---  
\author[Carlos D'Andrea]{Carlos D'Andrea}  
\address{Departament d'{\`A}lgebra i Geometria, Universitat de Barcelona.
Gran Via 585, 08007 Barcelona, Spain.}  
\email{cdandrea@ub.edu}  
\urladdr{\url{http://atlas.mat.ub.es/personals/dandrea/}}
  
%---  
\author[Mart{\'\i}n~Sombra]{Mart{\'\i}n~Sombra}  
\address{ICREA \&
Departament d'{\`A}lgebra i Geometria, Universitat de Barcelona.
Gran Via 585, 08007 Barcelona, Spain}
\email{sombra@ub.edu}
\urladdr{\url{http://atlas.mat.ub.es/personals/sombra/}}
  
\date{\today}  
\thanks{D'Andrea was partially supported 
by the research project~MTM2007--67493. 
Sombra was partially supported 
by the research projects~MTM2006-­14234 and MTM2009-14163-C02-01
 and by the visiting scholarship AGAUR 2007 BE-1 00036.}  
\begin{abstract}  
The Newton polygon of the implicit equation of a  
rational plane curve is explicitly determined by the multiplicities of any of  
its parametrizations.  
We give an intersection-theoretical proof of this fact based on a  
refinement of the Ku\v{s}nirenko-Bern\v{s}tein theorem.  
We apply this result to the determination of the Newton polygon of a curve  
parameterized by generic Laurent polynomials or by generic rational functions, with explicit genericity  
conditions.  
We also show that the variety of rational curves with given Newton polygon is unirational and we compute  its dimension.  
As a consequence, we  obtain  
that {\em any} convex lattice polygon with positive area is the  
Newton polygon of a rational plane curve.  
\end{abstract}  
  
\maketitle  
  
\overfullrule=0.3mm  
\parindent=0cm  
  
\vspace{-6mm}  
  
% ----------- Texto -----------------------------------------------  
\section{Introduction}

Rational plane curves can be represented either by a parametrization or  
by their implicit equation.  
In the present text we focus on the  
interplay between both representations.  
Specifically, we study  
the Newton polygon of the implicit equation of a rational  
plane curve, which turns out to be determined by the multiplicities of any  
given parametrization.  
We give a proof of this fact by translating the problem into the counting of the number of solution of a certain system of polynomial equations,  
which we solve by applying  
a refinement of the Ku\v{s}nirenko-Bern\v{s}tein theorem.  
As a consequence  of this result,  
we determine the Newton polygon of a curve  
parameterized by generic Laurent polynomials or by generic rational functions, with explicit genericity  
conditions.  
We also study the variety of rational curves with given Newton polygon; we show that this variety  is unirational and we compute  its dimension.  
  
\bigskip  
Let $\K$ be an algebraically closed field and set $  
\T^n:=(\K^\times)^n$ for the $n$-dimensional algebraic torus.  
For a plane curve $C\subset \T^2$ we consider its defining equation  
$E_C\in  
\K[{x}^{\pm1},{y}^{\pm1}]$, which is an irreducible Laurent polynomial well-defined up to a monomial factor.  
We will be mostly interested in the {\it Newton polygon of $C$}  
$$  
\newton(C)\subset\R^2,  
$$  
defined as the convex hull of the exponents in the monomial  
expansion of $E_C$.  
This is a lattice convex polygon in the sense that its vertexes lie  
in the lattice $\Z^2$; it is well-defined up  
to a translation.  
Note that $\newton(C)$ might reduce to a segment.  
  
On the other hand, let $f,\,g \in\K(t)^\times$  
be rational functions which are not simultaneously constant  
and consider the  map  
$$%\begin{equation}\label{phi}  
\rho:\T^1\dashrightarrow\T^2 \quad, \quad t \mapsto\left(f(t),g(t)\right).  
$$%\end{equation}  
The Zariski closure  
$\overline{\rho\left(\T^1\right)}$ of its image   is a rational curve in $\T^2$.  
We denote by $\deg(\rho)\in\N^\times$ the {\it degree of $\rho$},  
that is the cardinality of the fiber of $\rho$ above a  
generic point in its image.  
The notion of Newton polygon can be naturally extended to any effective  
Weil divisor $Z\in \Div(\T^2)$: given an equation  $E_Z\in \K[x^{\pm1},y^{\pm1}]$ for this divisor, the Newton polygon of $Z$  is just defined as the convex hull of the exponents of the monomials in~$E_Z$.  
By the definition of the push-forward cycle we then have  
$$  
\rho^*(\T^1)=\deg(\rho)\left[\ov{\rho(\T^1)}\right]  
$$  
  
Let $\P^1$ denote the projective line over $\K$, then for each  
$v\in\P^1$ we consider the {\it multiplicity of $\rho$ at $v$} defined as  
$$  
\ord_v(\rho):=\left(\ord_v(f),\ord_v(g)\right)\in\Z^2  
$$  
where $\ord_v(h)$ denotes the order of vanishing of $h$ at the point $v$.  
Note that  
$  
\ord_v(\rho)=(0,0)  
$ for all but a finite number of $v$'s and  
that the introduced family of vectors satisfies  
the {balancing condition} $$\sum_{v\in \P^1} \ord_v(\rho)=(0,0)  
$$  
because the sum of the order of the  
zeros and poles of a rational function is zero.  
  
\smallskip  
Given a family $B$ of vectors of $\Z^2$ which are zero except for a finite number of them and satisfy the balancing condition, we denote by  
$P(B)$ the lattice convex polygon  
constructed by rotating $-90$ degrees the non-zero vectors in $B$  
and concatenating them following their directions  
counterclockwise.  
Equivalently, $P(B)$ is characterized modulo translations by the properties that  
its inner normal directions are  
those spanned by the non-zero vectors in~$B$ and  
that for each such inner normal direction, the length of the corresponding edge  
of $P(B)$  
equals the sum of the lengths of the vectors in $B$ in that direction.  
  
\smallskip  
  
The following result gives the Newton polygon of a rational plane curve  
in terms of the multiplicities of any given parameterizations.  
  
\begin{theorem}\label{mt}  
Let $\rho:\T^1\dashrightarrow  \T^2$ be a non-constant rational map, then  
$$  
\newton(\rho^*(\T^1))=\deg(\rho) \newton\big(\ov{\rho(\T^1)}\big) = P\big((\ord_v(\rho))_{v\in \P^1}\big).  
$$  
\end{theorem}

\smallskip  
This result can be found in the work of A. Dickenstein, E.-M. Feichtner, B. Sturmfels and J. Tevelev~\cite{Tev07,DFS07,ST07},  
{\it see}~Remark~\ref{bernardo}.  
We illustrate it  with an example.  
  
\begin{example}\label{example:1}  
Consider the parametrization $\rho=(f,g)$ with  
$$  
f(t)=\frac{1}{t(t-1)} \quad, \quad  
 g(t)= \frac{t^2-5t+2}{t}.  
$$  
We have $\ord_0(\rho)= (-1,-1)$, $\ord_1(\rho)= (-1,0)$,  
$\ord_\infty(\rho)= (2,-1)$ and $\ord_{v_i}(\rho)= (0,1)$  
for each of the two zeros $v_1,v_2\in \T^1$  
of $t^2-5t+2$,  
while $\ord_v(\rho)=(0,0)$ for $v\ne 0,1,\infty, v_1,v_2$.  
Figure~\ref{fig:1} below shows these vectors and the polygon $\deg(\rho)\newton(\ov{\rho(\T^1)})$ obtained by rotating $-90$ degrees these vectors and  adequately  
concatenating them.  
\begin{figure}[htbp]
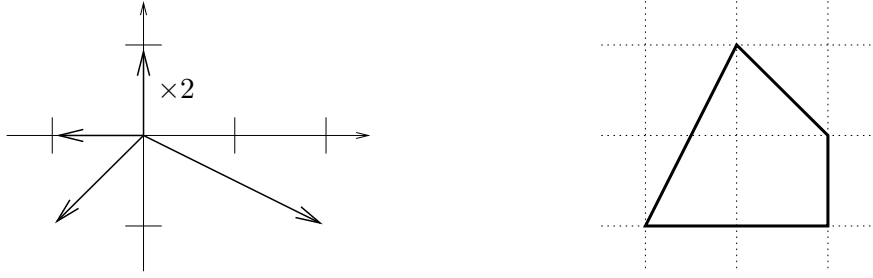
  
\input newton1.pstex_t  
\vspace{0mm}\caption{  
The multiplicities of $\rho$ and the Newton polygon of $\ov{\rho^*(\T^1)}$}\label{fig:1}  
\end{figure}  
Note that the resulting polygon is \textit{non-contractible} in the sense  
that it is not a translate of a non-trivial  
integer multiple of any other lattice polygon.  
This implies that $\deg(\rho)=1$ and so this is the Newton polygon of the curve $\ov{\rho(\T^1)}$.  
  
We can verify the obtained results by contrasting them with the actual  
equation of the curve:  
$$  
E_{\ov{\rho(\T^1)}} (x,y)=  
1-16x-4x^2-9xy-2x^2y -xy^2.  
$$  
\end{example}  
  
\medskip  
The computation of the Newton polygon  can always  
be done with no need of accessing  
the zeros and poles of $f$ and $g$. Indeed, suppose  
that $\F$ is a subfield of $\K$ and $f,g\in\F(t)$, then the  
polygon of $\newton(\rho^*(\T^1))$  
can be obtained from partial factorizations of the form  
$$  
f(t)= \alpha \prod_{p\in \cP} p(t)^{d_p}  
\quad,\quad  
g(t)= \beta \prod_{p\in \cP} p(t)^{e_p}  
$$  
for a finite set $\cP\subset \F[t]$ of pairwise coprime polynomials, $d_p,e_p\in \N$  
and $\alpha,\beta\in \F^\times$.  
This principle is already illustrated in the example above and explained for the general case in Lemma~\ref{practical}.  
Such factorizations can be computed  
with $\gcd$ operations only,  
by applying for instance  
the algorithm in~\cite[Lem.~6.2]{PS07}.  
It simplifies the computation of the Newton polygon in concrete examples,  
as it allows to perform  it  
within the field of definition of $\rho$.  
  
\smallskip  
In the present text we present an alternative proof of Theorem~\ref{mt}  
based on the  
refinement of the Ku\v{s}nirenko-Bern\v{s}tein estimate due to P.~Philippon and the second author~\cite{PS07a,PS07}.  
Instead of  
 explicitly dealing with the Newton polygon  we study its  
 support function  
$  
h(\newton(\rho^*(\T^1));\cdot):\R^2\to \R$.  
Set $p:=\Char(\K)$, we show that  
for  
$\sigma_1,\sigma_2\in \Z\setminus p  \Z$,  
the quantity  
$$  
h(\newton(\rho^*(\T^1)); (\sigma_1,\sigma_2))  
+h(\newton(\rho^*(\T^1)); (-\sigma_1,0))  
+h(\newton(\rho^*(\T^1)); (0,-\sigma_2))  
$$  
equals the number of solutions  
of the system of equations  
$$  
\ell_0+\ell_1\,x+\ell_2\,y=  
x^{\sigma_1}-f(t)=0  
=  
y^{\sigma_2}-g(t)=0  
$$  
for generic $\ell_0,\ell_1,\ell_2\in\K$, {\it see}  
Proposition~\ref{contador}.  
We solve this counting problem by applying the Philippon-Sombra estimate, which turns out to be exact  
for the present application.

\bigskip  
The study of the Newton polytope of a parameterized hypersurface  
is currently receiving a lot of attention.  
Part of the motivation is computational: {\it a priori}  
knowledge of  the Newton polytope  
restricts the possible monomials appearing in the implicit equation and thus reduces its computation to  
Numerical Linear Algebra.  
Reciprocally,  
there are algorithms which take as input  
a multivariate Laurent polynomial and profit  
from its  
monomial structure  
in order to decide more efficiently if the corresponding  
hypersurface admits a parametrization, and to find one when it does  
have~\cite{BS07}.  
  
The problem of computing the Newton polytope of a rational hypersurface  
was first posed in \cite{SY94} for  
generic Laurent polynomial parametrizations.  
The first results were obtained by I.~Emiris and I.~Kotsireas, who computed the Newton polygon and then the implicit equation of  
some specific parametrizations by  
determining the  
extremal monomials of a specialized resultant~\cite{EK03,EK05}.  
In~\cite{EKP10}, the case of rationally parameterized plane curves
with generic coefficients is 
treated by applying sparse elimination theory. 
Their main results are contained in our Corollaries~\ref{genericLaurent},
\ref{genericrationalequaldenom} and~\ref{genericrational} below.

\par\smallskip  
Motivation for this problem comes also from tropical geometry, as the  
data in the polytope of an effective divisor of $\T^n$ is equivalent  
to the abstract tropical variety associated with the divisor.  
In this direction,  
Dickenstein, Feichtner, Sturmfels and  Tevelev  
have  
determined the abstract tropical variety associated to  
a hypersurface parameterized by products of linear forms,  
which includes the case of curves~\cite{Tev07,DFS07,ST07}.  
Besides,  
Sturmfels, Tevelev and J.~Yu have computed the abstract tropical variety  
of a hypersurface parameterized by generic Laurent polynomials in any number of variables~\cite{STY07,ST07,SY07}.  
  
From another direction, A.~Esterov and A. Khovanski\u\i \  
found an important connection with combinatorics,  
as they  
showed  
that the Newton polytope of the projection of a generic complete intersection  
is isomorphic to the {mixed fiber polytope} of the Newton polytopes associated  
to the input data~\cite{EKho07}.  
  
\bigskip  
As an application of Theorem \ref{mt} above, we  
we  
recover the Newton polygon of the generic Laurent polynomial  
parametrization of dimension~$1$ obtained in~\cite{STY07} but this time  
with explicit genericity  
conditions. Proofs of the following statements can be found in Section \ref{degree}.  
  
\begin{corollary}[Generic Laurent Polynomials]\label{genericLaurent}  
Let $ D\ge d$ and $E\ge e$ and  
consider the parametrization $\rho=(p,q)$ where  
$$  
p(t)=\alpha_dt^d+\cdots+\alpha_Dt^D \quad, \quad  
q(t)=\beta_et^e+\cdots+\beta_Et^E \quad \in \K[t^{\pm1}]  
$$  
such that $\alpha_d,\alpha_D, \beta_e,\beta_E\ne 0$, then  
$$  
\deg(\rho)\newton\big(\ov{\rho(\T^1)}\big)  
= P\big((D-d,0),(0,E-e), (-D,-E),(d,e)\big)  
$$  
if and only if $\gcd(p,q)=1$.  
If moreover the vectors $(D-d,0),(0,E-e), (d,e)$ are not collinear, then  
$\deg(\rho)=1$ for generic $p,q$.  
\end{corollary}  
  
\medskip  
For parametrizations with generic rational functions with the same
denominator we find that $\newton(C)$ has at most five edges.  
  
\begin{corollary}[Generic Rational Functions with the Same  
Denominator]\label{genericrationalequaldenom}  
Let $ D\ge d$, $E\ge e$, $F\ge 0$ and consider the parametrization $\displaystyle \rho=\Big(\frac{p}{r},\frac{q}{r}\Big)\in \K(t)^2$ where  
$$  
p(t)=\alpha_dt^d+\cdots+\alpha_Dt^D \quad, \quad  
q(t)=\beta_et^e+\cdots+\beta_Et^E  
\quad, \quad  
r(t)=\gamma_0+\cdots+\gamma_Ft^F  
$$  
such that $\alpha_d,\alpha_D, \beta_e,\beta_E, \gamma_0,\gamma_F\ne 0$, then  
$$  
\deg(\rho)\newton\big(\ov{\rho(\T^1)}\big)  
= P\big((D-d,0),(0,E-e), (F-D,F-E),(d,e), (-F,-F)\big)  
$$  
if and only if $p,q,r$ are pairwise coprime.  
If moreover the vectors $(D-d,0),(0,E-e), (d,e),(F,F)$ are not collinear, then  
 $\deg(\rho)=1$ for generic $p,q,r$.  
\end{corollary}  
  
\medskip  
In the case of different denominators, the Newton polygon has at most six edges.

\begin{corollary}[Generic Rational Functions  
with Different Denominators]\label{genericrational}  
Let $ D\ge d$, $E\ge e$, $F,G\ge 0$, and consider the parametrization $\displaystyle  
\rho=\Big(\frac{p}{r},\frac{q}{s}\Big)\in \K(t)^2$ where  
$$  
p(t)=\alpha_dt^d+\cdots+\alpha_Dt^D \quad, \quad  
q(t)=\beta_et^e+\cdots+\beta_Et^E \quad \in \K[t^{\pm1}]  
$$  
and  
$$  
r(t)=\gamma_0+\cdots+\gamma_Ft^F \quad, \quad  
s(t)=\delta_0+\cdots+\delta_Gt^G \quad \in \K[t]  
$$  
such that $\alpha_d,\alpha_D, \beta_e,\beta_E, \gamma_0,\gamma_F, \delta_0,\delta_G\ne 0$, then  
$$  
\deg(\rho)\newton\big(\ov{\rho(\T^1)}\big)  
= P\big((D-d,0),(0,E-e), (F-D,G-E),(d,e), (-F,0), (0,-G)\big)  
$$  
if and only if $p,q,r,s$ are pairwise coprime.  
If moreover the vectors $(D-d,0),(0,E-e), (d,e), (F,0),(0,G)$  
are not collinear, then  
$\deg(\rho)=1$ for generic $p,q,r,s$.  
\end{corollary}

\medskip  
Newton polygons arising from the generic cases are very special: they have at most six edges and  a particular shape.  
It is then natural to ask which convex lattice polygons do realize  
as the  
Newton polygon of a rational plane curve.  
  
Let $Q\subset\R^2$ be an arbitrary lattice convex polygon.  
Setting $\tJ:=\Card(Q\cap\Z^2)-1$  
we identify  
the space of Laurent polynomials with support contained in $Q$ with  
$\K^{\tJ+1}$. Consider the set  
$$  
M_Q^\circ:=\big\{F\in\K[x^{\pm1},y^{\pm1}]: \ \newton(F)=Q , \ V(F) \subset \T^2 \mbox{  
is a rational curve} \big\}\subset \P^\tJ  
$$  
and let  
$M_Q  
\subset\P^\tJ$ denote its Zariski  
closure.  
The set $M_Q^\circ$ can be regarded as the space parametrizing divisors of $\T^2$ of the form $\delta[C]$  where $C$ is a rational curve, $\delta\ge 1$ and $\newton(\delta[C])=Q$.  
  
A convex lattice  
polygon of $\R^2$ is said {\it non-degenerate} if it is of dimension~$2$ or equivalently, if it has positive area.  
Recall that a variety is {\it unirational} if its function field  
admits a finite extension which is purely transcendental.  
As a further consequence of Theorem~\ref{mt} we obtain the following result.  
  
\begin{theorem}\label{M_Q}  
Let $Q\subset \R^2$ be a non-degenerate convex lattice polygon,  
then  
$M_Q$  
is a unirational  variety of dimension  
$\Card(\partial Q\cap\Z^2)-1.$  
If moreover $\Char(\K)=0$, the variety $M_Q$ is rational.  
\end{theorem}  
  
\medskip  
Under the assumption that $Q$ is non-degenerate,  
the polygon has at least three edges and so $\dim(M_Q)\ge 2$.  
In particular, $M_Q^\circ$ is non-empty. Moreover, we show in Proposition~\ref{rational map} that  
the generic member of $M_Q^\circ$ has multiplicity one or equivalently, that  
a generic parameterization corresponding to a divisor in $M_Q^\circ$ is  
birrational.  
In particular  
there always exists a rational plane curve  
$C$  
such that $Q=\newton(C).$

For a lattice segment  $S\subset \R^2$, we show in Proposition ~\ref{M_Q_segment} that $\dim(M_S)=1$ and that the multiplicity of a general member of $M_S^\circ$ equals $\ell(S)$.  
In particular, there  exists a rational plane curve  
$C$  
such that $S=\newton(C)$ if and only if $S$ does not contain any lattice point except its endpoints.  
  
\medskip  
We also characterize the polygons which can be realized as the Newton polygon  of a curve  
pa\-ra\-me\-trized with polynomials or with Laurent polynomials.  
  
\begin{theorem}\label{polly}  
Let $Q\subset\R^2$ be a non-degenerate lattice convex polygon, then  
$Q=\newton(\ov{\rho(\T^1)})$ for at least one $\rho\in \K[t]^2$  
({\it resp.} $\rho\in \K[t^{\pm1}]^2$) if and only if all but one ({\it resp.}  
one or two) of its inner normal directions lie in~$(\R_{\ge0})^2$.  
\end{theorem}  
  
\medskip  
This result gives some idea of the discrepancy between general rational  
curves and curves parameterized by polynomials or Laurent polynomials.  
For instance, for the polygon in Figure~\ref{fig:1} above,  the variety $M_Q$ is a rational  
hypersurface of $\P^5$ but none of its members is a curve parameterized with  
 polynomials or with Laurent polynomials.

\bigskip  
The text is organized as follows: in Section \ref{proof} we give a  
detailed proof of Theorem~\ref{mt}. In Section \ref{degree} we  
prove that if $Q$ is a non-degenerate polygon, then the generic  
parametrization having $Q=\newton(\rho^*(\T^1))$ is birational.  
We also give an algorithmic criteria in order to compute $\newton(C)$  
and finally prove Corollaries~\ref{genericLaurent},  
\ref{genericrationalequaldenom} and~\ref{genericrational}. In Section \ref{variedad}  
we study geometric properties of $M_Q$, and give  
proofs of Theorems~\ref{M_Q} and \ref{polly}. We conclude by considering the case when $Q$ is a lattice  
segment.

\par\bigskip{\bf Acknowledgments.}  
We thank Jos\'e Ignacio~Burgos, Laura Costa, Luis Felipe Tabera and Gerald Welters for helpful comments and conversations.  
We also thank Alicia Dickenstein and Bernd Sturmfels for pointing out  
and explaining us their results on tropical elimination theory.  
Experiments where carried out with the aid of  {\tt Maple} and  
M.~Franz's package {\tt Convex}~\cite{Fra06}.  
  
%Last but not least, we are also grateful to Julieta Sombra-Matiasich for  
%patiently waiting inside her mother's belly while his father  
%works on the present text.  

\bigskip  
\section{Proof of Theorem \ref{mt}}\label{proof}  
  
All considered varieties are defined over $\K$, reduced and  
irreducible. For a family of regular functions $f_1,\dots, f_s$ on  
an algebraic space, we denote by $V(f_1,\dots, f_s)$ the algebraic set  
they define in this space.  
A property depending on parameters  
is said {\it generic} if it holds for all points in a dense  
open subset of the  parameter space.

For a rational function $f\in \K(t)^\times$ we set $f_\num, f_\den\in \K[t]$ for its {\it numerator} and {\it denominator},  
which are coprime polynomials such that $f={f_\num}/{f_\den}$; these  
polynomials are well-defined up to a scalar factor.  
We denote by  
$$  
\deg(f):=\deg(f_\num)-\deg(f_\den)  
\quad ,\quad  
 \eta(f):=\max\{\deg(f_\num),\deg(f_\den)\}  
$$  
the {\it degree} and the {\it height} of $f$, respectively.  
  
A convex lattice  
polygon of $\R^2$ is {\it non-degenerate} if it is of dimension~$2$ or equivalently, if it has positive area.  
The {\it lattice length} $\ell(S)$  
of a lattice segment $S\subset \R^2$  is  
the number of  points of $\Z^2$ on it (including its endpoints) minus~$1$.  
We denote by $\N$ and $\N^\times$  
the set of  non-negative and positive integers respectively.

\bigskip  
With notation as in the introduction, we fix a reduced equation  
$E_C\in \K[x^{\pm1},y^{\pm1}]$ for  
the plane curve  $C=\ov{\rho(\T^1)} $ and we set $\newton(C):=\newton(E_C)\subset \R^2$ for its Newton polygon.  
A possible way of fixing $E_C$ (up to a non-zero scalar)  is to  
suppose that it lies in $\K[x,y]$ and that  
neither $x$ nor $y$ divide it or equivalently, that  
$\newton(C)$ is contained in  
the first quadrant and touches both the horizontal and vertical axes.  
Nevertheless,  
any other choice will be equally good for what follows.

\smallskip  
For $\sigma=(\sigma_1,\sigma_2)\in (\Z\setminus\{0\})^2$ we set  
$G(\sigma)$ for the Zariski closure of the set  
$$  
\{(t,x,y): x^{\sigma_1}=f(t), y^{\sigma_2}=g(t)\}\subset \T^1\times\T^2,  
$$  
which is a pure $1$-dimensional algebraic set.  
Let $\pi:\T^1\times\T^2\to \T^2$ denote the natural projection onto the second factor and  set  
$$  
C(\sigma):=\ov{\pi(G(\sigma))} \subset \T^2.  
$$  
Note that $G:=G(1,1)$ is the graph of $\rho$, while  
$C(1,1)=C.$  
  
\smallskip  
We define the {\em degree} $\deg(Z)$ of a pure $1$-dimensional algebraic set in the torus $\T^2$  
as its degree  
with respect to the standard  
inclusion $\iota:\T^2\hookrightarrow \P^2$;  in other words  
$$  
\deg(Z)=\Card\big( \iota(Z) \cap L\big)  
$$  
for a generic line $L\subset \P^2$.

Given a Laurent polynomial $F\in \K[x^{\pm1},y^{\pm1}]$  
we denote by $\deg_x(F)$ ({\it resp.} $\deg_y(F)$) its degree  
in the variable $x$ ({\it resp.} $y$), and $\deg(F)$ its total degree in $x$ and $y.$  
Its  
homogenization with respect to $\iota$ can then be written as $F^\h(w,x,y)=w^\delta F(x/w,y/w)$ with  
$$\delta = \deg(F(x,y))+ \deg_x(F(x^{-1},y))+ \deg_y(F(x,y^{-1})).$$  
Hence in case  $F$  is reduced, we have  
\begin{equation}\label{grado Laurent}  
\deg(V(F))= \deg(F(x,y))+ \deg_x(F(x^{-1},y))+ \deg_y(F(x,y^{-1})).  
\end{equation}  
  
\medskip  
The support function of a convex set $Q\subset \R^2$ is defined as  
$$  
h(Q;\cdot): \R^2\to \R \quad ,\quad  
w\mapsto \max\{\langle w,u\rangle,\, u\in Q\}.  
$$  
This is a piecewise affine convex function, which completely characterizes $Q$ as the set of points  
$u\in\R^2$ such that $\langle w,u\rangle \le h(Q;w)$ for all $ w\in \R^2$.  
Note that for $F\in\K[x^{\pm1},y^{\pm1}]$ such that $\newton(F)=Q$ and  
$\sigma=(\sigma_1,\sigma_2)\in\Z$ we have  
$\deg\left(F(x^{\sigma_1},y^{\sigma_2})\right)=h(Q;\sigma)$. This fact will be used in the sequel.  
  
\medskip  
The following result expresses  
a linear combination of support functions of $\newton(C)$ as an intersection number.  
It can be regarded as some kind of extension of the ``Perron's  
theorem'' in~\cite[Thm.~3.3]{Jel05}.

\begin{proposition}\label{contador}  
With notation as above, set $p:=\Char(\K)$ and let $\sigma_1,\sigma_2  
\in \Z\setminus p \Z$, then  
$\deg(\pi|_{G(\sigma)})=\deg(\rho)$ and  
\begin{equation}\label{formula}  
h(\newton(C); (\sigma_1,\sigma_2))  
+h(\newton(C); (-\sigma_1,0))  
+h(\newton(C); (0,-\sigma_2))  
=\deg(C(\sigma_1,\sigma_2)).  
\end{equation}  
\end{proposition}  
  
\begin{proof}  
Let $\pi:\T^1\times \T^2\to \T^2$ denote the natural projection onto the second  
factor and $\chi:\T^2\to \T^2$ the monomial map $(x,y)\mapsto(x^{\sigma_1},y^{\sigma_2})$.  
Set $\sigma:=(\sigma_1,\sigma_2)$ and consider the commutative diagram  
\begin{equation}\label{diagrammereduc}  
\xymatrix{  
G(\sigma) \ar[r]^{\Id_{\T^1}\times \chi }  
\ar[d]^{\pi} & G\ar[d]^{\pi}\\  
C(\sigma) \ar[r]^{\chi} & C}  
\end{equation}  
The horizontal arrows are finite coverings of the same  
degree $|\sigma_1\sigma_2|$. On the other hand  
$$  
\deg(\chi) \deg(\pi|_{G(\sigma)}) =  
\deg(\chi \circ \pi|_{G(\sigma)}) =  
\deg\big(\pi|_{G} \circ (\Id_{\T^1}\times \chi)\big)  
= \deg(\pi|_{G}) \deg(\Id_{\T^1}\times \chi) ,  
$$  
which implies the first part of the proposition: $\deg(\pi|_{G(\sigma)})=  
\deg(\pi|_{G}) =\deg(\rho)$.  
  
\medskip  
For the second part, we claim that  
$\chi^*(E_C)=E_C(x^{\sigma_1},y^{\sigma_2})$  
is a reduced equation for~$C(\sigma)$.  
We recall that a Laurent polynomial $F\in \K[x^{\pm1},y^{\pm1}]$ is reduced if and only if  
its gradient $\nabla(F) =(\partial F /\partial x, \partial F/\partial y)$ does not vanish on any of the components of $V(F)$, see for instance ~\cite[Lem.~3.1]{Jel05}.  
This condition automatically holds for $\chi^*(E_C)$, because by construction  
it holds for  $E_C$, and  
$\sigma_1,\sigma_2 \notin p\Z$. Hence $\chi^*(E_C)$ is reduced.  
On the other hand, it is clear that $C(\sigma)$ coincides with the zero set of  
$\chi^*(E_C)$, which proves the claim.  
  
\smallskip  
Identity~(\ref{grado Laurent}) implies then  
\begin{align*}  
\deg(C(\sigma))& = \deg(\chi^*(E_C)(x,y))  
+ \deg_x(\chi^*(E_C)(x^{-1},y))+ \deg_y(\chi^*(E_C)(x,y^{-1})) \\[2mm]  
& = \deg(E(x^{\sigma_1},y^{\sigma_2}))  
+ \deg_x(E(x^{-\sigma_1},y^{\sigma_2}))+ \deg_y(E(x^{\sigma_1},y^{-\sigma_2})) \\[2mm]  
& = h(\newton(C); (\sigma_1,\sigma_2))  
+h(\newton(C); (-\sigma_1,0))  
+h(\newton(C); (0,-\sigma_2)).  
\end{align*}  
\end{proof}

\smallskip  
  
Note that the algebraic set $G(\sigma)$ can be written as  
$$  
G(\sigma)=V\big(f_\den(t)x^{\sigma_1}-f_\num(t),g_\den(t)y^{\sigma_2}-g_\num(t)\big).  
$$  
For a generic line $L\subset \P^2$, we have that  
$G(\sigma)\cap (\T^1\times\iota^{-1}(L)) = \pi^{-1}(C(\sigma)\cap L)$  
and so  
$$  
\Card\big(G(\sigma)\cap (\T^1\times\iota^{-1}(L))\big) = \deg(\pi|_{G(\sigma)}) \deg(C(\sigma)).  
$$  
Hence Proposition~\ref{contador}  implies that for  
$\sigma_1,\sigma_2\in \Z\setminus p \Z$,  
the quantity  
$$  
\deg(\rho)\big(h(\newton(C); (\sigma_1,\sigma_2))  
+h(\newton(C); (-\sigma_1,0))  
+h(\newton(C); (0,-\sigma_2)) \big)  
$$  
equals the number of solutions in $\T^1\times\T^2$ of the system of equations  
\begin{align} \label{system}  
\nonumber \ell_0+\ell_1\,x+\ell_2\,y&=0\\[1mm]  
f_\num(t)-f_\den(t)x^{\sigma_1}  
&=0\\[1mm]  
\nonumber  
g_\num(t)-g_\den(t)y^{\sigma_2}  
&=0  
\end{align}  
for generic $\ell_0,\ell_1,\ell_2\in\K$.  
  
\medskip  
The number of solutions of this system  
can be expressed in combinatorial terms thanks to a result of Philippon and the second author~\cite{PS07a,PS07}.  
We introduce some combinatorial  
invariants in order to explain it better.  
Restricting to our setting, let  
$$  
H(t,x,y)=\sum_{j=0}^M\alpha_j(t) x^{a_j}y^{b_j}\in\K(t)[x^{\pm 1},y^{\pm 1}]  
$$  
be a non-zero Laurent polynomial in the variables $x,y$ with coefficients rational functions in the variable $t$.  
For each $v\in \P^1$ consider the {\it $v$-adic Newton polytope of~$H$}  
$$  
\newton_v(H):=\Conv\big( (a_0,b_0,-\ord_v(\alpha_0)), \dots,(a_M,b_M,-\ord_v(\alpha_M))\big) \subset \R^{3}.  
$$  
This polytope sits above the usual Newton polygon  
$$\newton(H):=\Conv((a_0,b_0),\dots, (a_M,b_M))\subset \R^2$$  
{\it via} the natural projection $\R^{3}\to \R^2$ which forgets the last coordinate.  
Consider the {\em roof function of $H$ at $v$} defined as  
$$  
\vartheta_v(H):\newton(H)\rightarrow\R\quad,\quad {(x,y)}\mapsto  
\max\{z\in\R:({x,y,}z)\in\newton_v(H)\}$$  
that is,  the concave and piecewise affine function parameterizing the upper envelope of $\newton_v(H)$ above $\newton(H)$.  
  
\smallskip  
For the polynomials in the system~(\ref{system}), the respective Newton polygons are  
\begin{align*}  
P_0&:=\newton(\ell_0+\ell_1x+\ell_2)= \Conv\big((0,0),(1,0),(0,1)\big), \\[1mm]  
P_1&:=\newton(f_\num(t)-f_\den(t)x^{\sigma_1}  
)= \Conv\big((0,0), (\sigma_1,0)\big) , \\[1mm]  
P_2&:=\newton(g_\num(t)-g_\den(t)y^{\sigma_2})= \Conv\big((0,0), (0,\sigma_2)\big),  
\end{align*}  
so that $P_0$ is the standard triangle in $\R^2$ while $P_1$ and $P_2$ are segments.  
For $v\in \P^1$ and $i=0,1,2$ we denote $P_{i,v} \subset \R^3$ and  
$\vartheta_{i,v}:P_i\to \R$  the $v$-adic Newton polytope and  
corresponding roof function for the polynomials  
$\ell_0+\ell_1\,x+\ell_2\,y$,  
$f_\den(t)x^{\sigma_1}-f_\num(t)$ and $g_\den(t)y^{\sigma_2}-g_\num(t)$,  
respectively.  
Computing them explicitly, we get $\vartheta_{0,v}\equiv 0$, while for $i=1,2$ we have that  
$\vartheta_{i,v}$ is the affine function on $P_i$ such that  
$$  
\vartheta_{1,v}(0,0) = -\ord_v(f_\num) \quad ,\quad  
\vartheta_{1,v}(\sigma_1,0)=-\ord_v(f_\den)  
$$  
and  
$$  
\vartheta_{2,v}(0,0) = -\ord_v(g_\num) \quad ,\quad  
\vartheta_{2,v}(0,\sigma_2)=-\ord_v(g_\den).  
$$  
Note that $\vartheta_{i,v}\le 0$ for $v\ne \infty$ while $\vartheta_{i,v}\ge0$ for $v=\infty$.  
The figure below shows the graph of these roof functions.  
  
\begin{figure}[htbp]
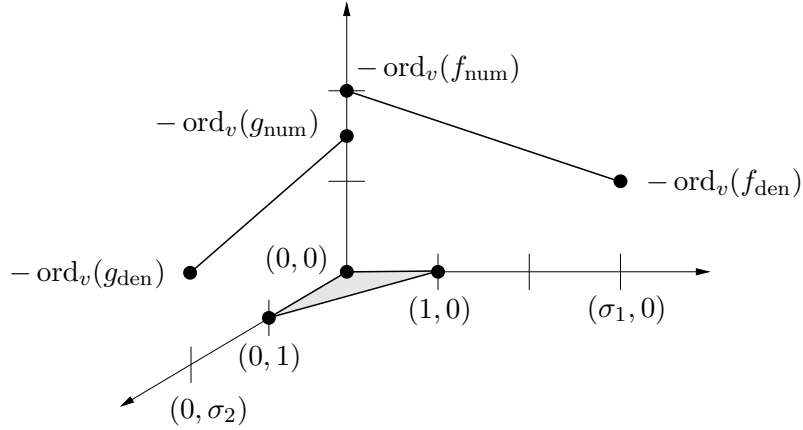
  
\input techos.pstex_t  
\vspace{-3mm}\caption{The $v$-adic Newton polytopes for the system~(\ref{system})}\label{fig:2}  
\end{figure}  
  
As the polynomials in~(\ref{system}) are primitive we can  
apply~\cite[Thm.~1.2]{PS07}, which shows that the number of solutions is bounded above by the quantity  
\begin{equation}\label{KB}  
\sum_{v\in \P^1}\MI_2(\vartheta_{0,v},\vartheta_{1,v},\vartheta_{2,v}).  
\end{equation}  
The {mixed integral} $\MI$ in this expression  
is the natural extension to concave functions of the  
mixed volume of convex bodies. As such, it satisfies  
analogous properties.  
We refer to~\cite[\S~IV]{PS04} and~\cite[\S~8]{PS07} for its definition and basic properties.  
  
\smallskip  
The estimate (\ref{KB}) is exact in this case because none  
of the relevant \textit{initial systems} has a root.  
Initial systems  in this context can be interpreted as the restriction of the input system  
to the face of the upper envelope of the $v$-adic Newton polytopes  
$P_{i,v}$'s corresponding to $\tau$. We will explain them briefly here, but refer the reader interested in more details to ~\cite[\S~6]{PS07}.  
\par  
Let $H\in \K(t)[x^{\pm1}, y^{\pm1}]$ and  $\tau\in \R^2$, then  
for $v\in \P^1\setminus \{\infty\}$ the {\it $\tau$-initial part of $H$ at $v$}  
is defined as the Laurent polynomial  
$\init_{v,\tau}(H)\in\K[x^{\pm1},y^{\pm1}]$ such that  
$$  
H(t, t^{\tau_1}x,t^{\tau_2}y) = (t-v)^c (\init_{v,\tau}(H)(x,y) + o(1))  
$$  
for a $c\in\Z$ and $o(1)$ going to $0$ as $t\to v$, while the {\em $\tau$-initial part of $H$ at $\infty$} is just defined as the $\tau$-initial part of $H(t^{-1},x,y)$ at $0$.  
  
\par\smallskip  
By~\cite[Prop.~1.4]{PS07}, if  
for all $v\in \P^1\setminus \{\infty\}$ and $\tau \ne (0,0)$, and for  
$v=\infty$ and all $\tau\in \R^2$,  
the system of equations  
\begin{equation*}  
\init_{v,\tau}(\ell_0+\ell_1x+\ell_2)=  
\init_{v,\tau}(f_\num(t)-f_\den(t)x^{\sigma_1})=  
\init_{v,\tau}(g_\num(t)-g_\den(t)y^{\sigma_2})=0  
\end{equation*}  
has no solution  
in $\T^2$, then the estimate~(\ref{KB}) counts exactly  
 the number of roots of the system in $\A^1\times \T^2$.  
%It can be shown that  
%all of their roots lie in fact in $\T^1\times \T^2$,  
%for a  generic linear form $\ell_0+\ell_1x+\ell_2y$.  
  
%\smallskip  
  
In our setting, at least one of the initial polynomials  
\begin{equation}\label{initial}  
\init_{v,\tau}(f_\num(t)-f_\den(t)x^{\sigma_1}) \quad, \quad  
\init_{v,\tau}(g_\num(t)-g_\den(t)y^{\sigma_2})  
\end{equation}  
reduces to  a monomial (and hence the initial system has no solution in $\T^2$) unless  
$$  
\tau=\tau(v):= \Big(-\frac{\ord_v(f)}{\sigma_1}, -\frac{\ord_v(g)}{\sigma_2}\Big).  
$$  
We have that $\tau(v)=(0,0)$ for all but a finite number of $v$'s  
and by the stated equality criterion these $v$'s need not be considered.  
On the other hand,  
the finite number of $v$'s  such that $\tau(v)\ne (0,0)$  
produces a finite number of solutions for the initial system~(\ref{initial}),  
all of which are avoided by the linear form $\ell_0+\ell_1x+\ell_2y$, as this last one is supposed generic.

Hence none of the relevant initial systems have solutions and so~(\ref{KB}) is an equality, as announced.  
We conclude that for $\sigma_1,\sigma_2\in \Z\setminus p\Z$ the degree of the curve $C(\sigma)$ can be expressed as  
\begin{equation}\label{KB exacto}  
\deg(C(\sigma))= \frac{1}{\deg(\rho)}  
\sum_{v\in \P^1}\MI_2(\vartheta_{0,v},\vartheta_{1,v},\vartheta_{2,v}).  
\end{equation}  
  
We extend this formula to $\sigma\in (\R^\times)^2$ and make it  
explicit by computing the relevant mixed integrals.  
  
\begin{proposition}\label{MI}  
With notation as before, let $\sigma\in (\R^\times)^2$, then  
\begin{align}\label{formulaMI}  
h(\newton&(C); (\sigma_1,\sigma_2))  
+h(\newton(C); (-\sigma_1,0))  
+h(\newton(C); (0,-\sigma_2)) \\[2mm]  
&= \frac{1}{\deg(\rho)}\sum_{v\in \P^1} \max\big\{0,  
-\ord_v(g)\, \sign(\sigma_1\sigma_2) \sigma_1,  
-\ord_v(f)\, \sign(\sigma_1\sigma_2) \sigma_2 \big\}. \nonumber  
\end{align}  
\end{proposition}  
  
\begin{proof}  
Suppose first $\sigma_1,\sigma_2\in \Z\setminus p\Z$.  
We compute the mixed integrals in Formula~(\ref{KB})  
by applying the decomposition formula~\cite[Formula~(8.6)]{PS07}: with notation as before, for each  $v \in \P^1$ we have  
\begin{align*}  
\MI_2(\vartheta_{0,v},\vartheta_{1,v},\vartheta_{2,v}) =&  
\sum_{u\in S^1} h({P_0};u) \MI_1( \vartheta_{1,v}|_{P_1^u}, \vartheta_{2,v}|_{P_2^u}) \\[0mm]  
& +  
\sum_{y\in S^2_+} h(P_{0,v};y) \MV_2(P_{1,v}^y,P_{2,v}^y).  
\end{align*}  
Here $S^1 \subset \R^2$  denotes the unit circle and $S^2_+\subset \R^2$ denotes the set of points in the sphere $S^2$ whose last coordinate is positive, while  
 $P_i^u$  ({\it resp.} $P_{i,v}^y$) stands for the face ({\it resp.} the slope)  
in the direction $u$ ({\it resp.} $y$) of $P_i$ ({\it resp.} $P_{i,v}$).

\medskip  
For each $v\in\P^1,$ set $I_v$ for the first sum. We have  
$$  
I_v= h(P_0;(1,0)) \MI_1\big(\vartheta_{1,v}|_{P_1^{(1,0)}}, \vartheta_{2,v}\big)  
+h(P_0;(0,1)) \MI_1\big(\vartheta_{1,v}, \vartheta_{2,v}|_{P_2^{(0,1)}}\big) .  
$$  
It turns out that $P_1^{(1,0)}$ must be equal to one of the points  
$(\sigma_1,0)$ or $(0,0)$ depending on the sign of $\sigma_1$. Similarly,  
${P_2^{(0,1)}}$ must be equal to either $(0,\sigma_2)$ or $(0,0)$ depending on the sign of $\sigma_2$.  
By~\cite[Formula~(8.3)]{PS07}  
$$  
I_v= |\sigma_1| \, \vartheta_{1,v}({P_1^{(1,0)}})+  
|\sigma_2|\, \vartheta_{2,v}({P_2^{(0,1)}})  
$$  
and so $\sum_v I_v=0$, because $\vartheta_{i,v}({P_i^{e_i}})$ equals the order of vanishing  
at $v$ of a rational function $q\in \K(t)^\times$,  
and the identity  
$\sum_{v\in \P^1} \ord_v(q)=0$ always holds.  
  
\medskip  
Set $J_v$ for the second sum, which consists in the only term corresponding to  
the vector $y\in S^2_+$ such that $P_{1,v}^y= P_{1,v}$ and $P_{2,v}^y=P_{2,v}$.  
Setting $V:= (\sigma_1,0,\ord_v(f))$ and  
$W:= (0,\sigma_2,\ord_v(g))$, their exterior product equals  
$$  
V\times W= (- \ord_v(f)\sigma_2, -\ord_v(g)\sigma_1, \sigma_1\sigma_2),  
$$  
and so the only relevant $y$ is  
$$  
y=  
\frac{ \sign(\sigma_1\sigma_2) }{ ||V\times W ||} \ V\times W.  
$$  
We have then $\MV_2(P_{1,v}^y,P_{2,v}^y)=||V\times W||_2$, hence  
\begin{align*}  
J_v&=h(P_0;y) \MV_2(P_{1,v}^y,P_{2,v}^y)\\[2mm]  
&= \max\{0,-\ord_v(f) \sign(\sigma_1\sigma_2) \sigma_2  
, -\ord_v(g) \sign(\sigma_1\sigma_2) \sigma_1  
\}.  
\end{align*}  
  
Finally, note that both sides of the identity are  continuous and  
homogeneous with respect to homotheties, which implies the general case  
 $\sigma_1,\sigma_2\in \R^\times$.  
\end{proof}

For each $v\in \P^1$ consider a rectangular triangle $R_v$ defined as  
$$  
R_v:=  
\Conv\big( (0,0), (-\ord_v(g),0),(0,-\ord_v(f))\big)  
$$  
if $\ord_v(f)\ord_v(g)\ge 0$ and  
$$  
R_v:=\Conv\big( (-\ord_v(g),-\ord_v(f)), (-\ord_v(g),0),(0,-\ord_v(f))\big)  
$$  
if $\ord_v(f)\ord_v(g)\le 0 $.  
Both definitions coincide in case $\ord_v(f)\ord_v(g)=0$.  
Note that $R_v=(0,0)$ for all but a finite number of $v$'s.  
The following result gives an explicit expression for the Minkowski difference of  
$\newton(C)$ with the minimal rectangle containing it.  
  
\begin{proposition}\label{Q-R}  
With the above notation, for $i=1,2$  
let $\pi_i:\R^2\to \R$ denote the projection to the $i$-th factor, then  
$$  
\newton(C)-\pi_1(\newton(C))\times\pi_2(\newton(C)) = \frac{1}{\deg(\rho)}  
\sum_{v\in \P^1} R_v .  
$$  
\end{proposition}

\begin{proof}  
Set for short  
$$  
D:= \deg(\rho) \big( \newton(C)-\pi_1(\newton(C))\times\pi_2(\newton(C))\big) \subset \R^2.  
$$  
It suffices to show that the support function of both $D$ and  
 $\sum_{v} R_v $  
coincide for all  $\sigma\in (\R^\times)^2$.  
  
\smallskip  
In case  
$\sigma_1\sigma_2>0$ set  
$T_v:=\Conv \big( (0,0),  
 (-\ord_v(g),0), (0,-\ord_v(f)) \big)$, so that  
Proposition~\ref{MI} implies that  
$$  
h(D;\sigma) = \sum_v h(T_v; \sigma).  
$$  
When $\ord_v(f)\ord_v(g)\ge 0 $ we have that $T_v=R_v$,  
while if $\ord_v(f)\ord_v(g)\le 0 $, the maximum of the scalar product of $\sigma$  
over the rectangle $R_v\cup T_v$ is attained  
either at the point $(-\ord_v(g),0)$ or at $(0,-\ord_v(f))$,  
hence also in this case $h(T_v;\sigma)=h(R_v;\sigma)$.  
Thus  
$$  
h(D;\sigma) = \sum_v h(R_v; \sigma).  
$$  
\smallskip  
In case $\sigma_1\sigma_2<0$, we use the identities  
$$  
\sum_v\ord_v(f) = \sum_v \ord_v(g) =0,  
$$  
For $v\in \P^1$ set  
$$  
U_v:= \Conv\big(  
 (-\ord_v(g),-\ord_v(f)), (-\ord_v(g),0),(0,-\ord_v(f))\big),  
$$  
Proposition~\ref{MI} then implies that  
\begin{align*}  
h(D;\sigma)& = \sum_v \max \big\{ \langle\sigma , (0,0)\rangle,  
\langle\sigma , (\ord_v(g),0)\rangle,\langle\sigma, (0,\ord_v(f))  
 \rangle \big\} \\  
&\hspace{5mm}- \sum_v \langle \sigma, (\ord_v(g),\ord_v(f))\rangle \\[1mm]  
&=  \sum_v \max\big\{ \langle\sigma , (-\ord_v(g),-\ord_v(f))\rangle,  
\langle\sigma , (-\ord_v(g),0)\rangle,\langle \sigma, (0,-\ord_v(f))\rangle \big\}\\[1mm]  
&=  \sum_v h(U_v;\sigma).  
\end{align*}  
The rest of the argument is as in the preceding case:  
when $\ord_v(f)\ord_v(g)\le 0 $ we have that $U_v=R_v$,  
while if $\ord_v(f)\ord_v(g)\ge 0 $, the maximum of the scalar product of $\sigma$  
over the rectangle $R_v\cup U_v$ is attained  
either at the point $(-\ord_v(g),0)$ or at $(0,-\ord_v(f))$,  
hence also in this case $h(U_v;\sigma)=h(R_v;\sigma)$.  
We conclude that for all    $\sigma\in (\R^\times)^2$  
$$  
h(D;\sigma) = \sum_v h(R_v; \sigma),  
$$  
while on the other hand  
$  
h(\sum_v R_v ;\sigma)= \sum_v h(R_v ;\sigma)  
$  
by~\cite[Thm.~1.7.5, p.~41]{Sch93}, which concludes the proof.  
\end{proof}  
  
\begin{example}  
Consider the rational functions  
$$  
f(t):= \frac{t(t-2)}{(t-1)^2} \quad, \quad g(t):=\frac{t(t-1)^2}{(t-2)^3}.  
$$  
We have $\ord_0(f,g)=(1,1)$, $\ord_1(f,g)=(-2,2)$, $\ord_2(f,g)=(1,-3)$ and  
$\ord_v(f,g)=(0,0)$ for $v\neq 0,1,2$.  
Figure~\ref{fig:3} below illustrates the sum of the corresponding rectangular triangles.  
The dashed lines indicate the  
Newton polygon of the implicit equation for this example.  
  
\begin{figure}[htbp]
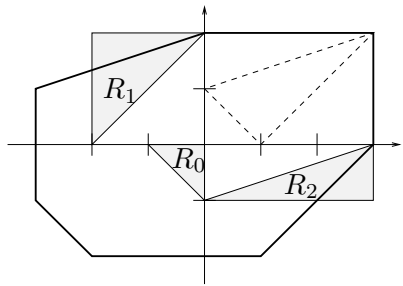
  
\input minkowski.pstex_t  
\vspace{-3mm}\caption{The Minkowski sum of the $R_v$'s}\label{fig:3}  
\end{figure}  
\end{example}  
  
The following lemma gives the minimal rectangle containing  
$\newton(C)$ and it can be found for instance,  
in~\cite[Thm.~6]{SW01}. It is also a consequence of  
Proposition~\ref{MI} above and for convenience of the reader we  
give its proof. For simplicity, we fix $E_C$ up to a non-zero  
scalar, by supposing it lies in $\K[x,y]$ and that neither $x$ nor  
$y$ divide it.  
  
\begin{lemma}\label{gradoparcial}  
With the above notation, we have  
$$\deg(\rho)\,\deg_x(E_C)= \eta(g) \quad ,\quad  
\deg(\rho)\,\deg_y(E_C)= \eta(f).  
$$  
\end{lemma}  
  
\begin{proof}  
With the given choice of $E_C$, we have that  
$$\deg_x(E_C)= h(\newton(C);(1,0) ) +  h(\newton(C); (-1,0))+  h(\newton(C); (0,0)).  
$$  
By Proposition~\ref{MI}  
applied to $\sigma_1=1$ and $\sigma_2\to 0^+$  
we obtain  
\begin{align*}  
{\deg(\rho)}\deg_x(E_C) & =  
\sum_{v\in \P^1} \max\{0, - \ord_v(g)\} \\[1mm]  
&= \max\{0, - \ord_\infty(g)\} +\sum_{v\in \A^1} \max\{0, - \ord_v(g)\}  
\\[1mm]  
&= \max\{0,\deg(g_\num)-\deg(g_\den)\}+ \deg(g_\den)  
 \\[2mm]  
&=\max\{ \deg(g_\num),\deg(g_\den)\}=\eta(g).  
\end{align*}  
The expression for $\deg_y(E_C)$ follows similarly.  
\end{proof}  
  
This lemma is equivalent to  
$$  
\pi_1(\newton(C)) = \frac{1}{\deg(\rho)}  
\Conv(0, \eta(g)) \quad ,\quad  
\pi_2(\newton(C))  = \frac{1}{\deg(\rho)}  
\Conv(0,\eta(f)).  
$$  
Now that we know the size of this minimal rectangle, we are in position to compute $\newton(C)$ by ``extracting'' it from the expression in Proposition~\ref{Q-R}.  
  
\medskip  
\begin{Proof}{Proof of theorem~\ref{mt}.}  
As a first step, we determine the edge structure of the polygon  
$$  
R:=\sum_{v\in \P^1} R_v .  
$$  
Let $v\in \P^1$. In case  
$\ord_v(f)\ord_v(g)>0$  
we have that  
$$  
\langle\ord_v(\rho), (-\ord_v(g),0) \rangle  
= \langle\ord_v(\rho), (0, -\ord_v(f)) \rangle  
= -\ord_v(f)\ord_v(g) < 0  
$$  
while $\langle\ord_v(\rho), (0,0) \rangle=0$.  
Hence for each  $v$, the rectangular triangle $R_v$ contributes to $R$  
with  
\begin{itemize}  
\item  one edge of length $||\ord_v(\rho)||$ and inner normal  
$\ord_v(\rho)$;  
\smallskip  
\item one vertical  edge of length $|\ord_v(g)| $  
and inner normal $(\sign(-\ord_v(f)), 0)$;  
\smallskip  
\item one horizontal edge of length $|\ord_v(f)| $ and inner normal  
$(0,\sign(-\ord_v(g)))$.  
\end{itemize}  
  
\smallskip  
In case  
$\ord_v(f)\ord_v(g)<0$,  
with a similar computation we can verify that $R_v$  contributes to $R$  
with  
\begin{itemize}  
\item  one edge of length $||\ord_v(\rho)||$ and inner normal  
$\ord_v(\rho)$;  
\smallskip  
\item one vertical  edge of length $|\ord_v(g)| $  
and inner normal $(\sign(-\ord_v(f)), 0)$;  
\smallskip  
\item one horizontal edge of length $|\ord_v(f)| $ and inner normal  
$(0,\sign(-\ord_v(g)))$;  
\end{itemize}  
exactly as in the first case.  
  
\smallskip  
If $\ord_v(f)=0$ and $\ord_v(g) \ne 0$, the triangle $R_v$ contributes to $R$  
with  
two horizontal edges of length $|\ord_v(g)|$, with inner  
normal $(0,1)$ and $(0,-1)$ respectively.  
Similarly, in case  
$\ord_v(g)=0$ and $\ord_v(f) \ne 0$,  the triangle $R_v$ contributes  
with  
two vertical edges of length $|\ord_v(f)|$, with inner  
normal $(1,0)$ and $(-1,0)$ respectively.  
  
\medskip  
Next we list the edges of $R$ as its inner normal runs over $S^1$:  
for $\theta\in S^1$, $\theta\ne (\pm1,0), (0,\pm1)$ as we saw above,  
the corresponding edge has length $\sum_v\|\ord_v(\rho)\|$,  
the sum being over all $v$ such that $\ord_v(\rho)\in\left(\R_{>0}\right)\theta$.  
\par

For $\theta=(0,1)$, the length of the corresponding horizontal edge of $R$  
equals  
$$  
\sum_{v : \ \ord_v(f)>0} -\ord_v(g)  
 + \sum_{v:\ \ord_v(g) =0,  
-\ord_v(f)<0} \ord_v(g) .  
$$  
By Lemma~\ref{gradoparcial},  
the corresponding edge in  
$-\deg(\rho) \big(  
\pi_1(\newton(C))\times\pi_2(\newton(C))\big)$  
has length  
$  
\sum_{v : \ \ord_v(f)>0} -\ord_v(g)$.  
This shows that the horizontal edge of $R$ with normal $(0,1)$ has length  
$$  
\sum_{v:\ \ord_v(g) =0,  
-\ord_v(f)<0} \ord_v(g)  
$$  
as stated.  
The remaining cases ($\theta=(\pm1,0), \theta=(0,-1)$) are treated similarly.  
\par\smallskip  
To conclude, we use the fact that the edges of the Minkowski sum of two polygons in $\R^2$, are exactly  
the concatenation  
of the edges of the given polygons, {\it see} the discussion in~\cite[pp.~144-145]{Sch93}.  
\end{Proof}  
  
\begin{remark} \label{bernardo}  
Theorem~\ref{mt} can easily be rephrased  
in the language of  
tropical geometry. Given an effective cycle $Z$ of $\T^n$, its  
{abstract tropical variety} is  
a pair $(\cT_Z,m_Z)$ where $\cT_Z\subset\R^n$ denotes the tropicalization of $Z$  
with respect to the trivial valuation of~$\K$  
and $m_Z:\cT^0_Z\to \N\setminus\{0\}$ is a locally constant  
function defined on the set of regular points $\cT^0_Z$ of $\cT_Z$, {\it see}~\cite{ST07}  
for the precise definitions.  
For an effective divisor $Z\in \Div(\T^2)$  
this data is equivalent to the Newton polygon: the tropicalization  
$\cT_Z$ consists in the union of the inner normal directions of $\newton(Z)$,  
$\cT^0_Z=\cT_Z\setminus\{(0,0)\}$ and for a point $b\ne (0,0)$ in one of these directions, $m_Z(b)$ equals the lattice length of the corresponding edge.  
Theorem~\ref{mt} is then equivalent to  
$$  
\cT_{\rho^*(\T^1)}=  
\bigcup_{v\in\P^1} (\R_{\ge0})\, \ord_v(\rho),  
$$  
and  
$  
m_{\rho^*(\T^1)}(b)= \sum_{v: \ \ord_v(\rho)\in (\R_{>0})b} \ell(\ord_v(\rho))$  
for  $b\in \cT_{\rho^*(\T^1)}^0$.  
In this form, the result can be found in the work of Dickenstein, Feichtner, Sturmfels and Tevelev: the tropicalization of $\rho^*(\T^1)$ is given by~\cite[Thm. 3.1]{DFS07} or~\cite[Thm. 3.1]{Tev07} while the corresponding  multiplicities can be obtained from~\cite[Thm.~1.1]{ST07}, {\it see} also~\cite[Rem.~3.2]{DFS07}.  
\end{remark}

%\bigskip  
\section{The Generic Cases}\label{degree}  
  
In this section we apply our main result to the computation of the Newton polygon in the generic cases  
(Corollaries~\ref{genericLaurent},  
\ref{genericrationalequaldenom} and~\ref{genericrational} in the introduction).  
The following result  
gives the degree of a generic parametrization and  
is an  
important ingredient  
in the proof of these corollaries.  
  
\begin{proposition}\label{grado generico}  
Let $(m_1,n_1),\dots, (m_R,n_R)\in \Z^2$  
and for  $\alpha,\beta\in \K^\times$  and $v_1,\dots, v_R\in \K$ set  
\begin{equation}\label{kk}  
f(t):=\alpha\prod_{i=1}^M(t-v_i)^{m_i}  
\quad, \quad  
g(t):=\beta\prod_{i=1}^M(t-v_i)^{n_i}.  
\end{equation}  
If there is a pair $1\le i\neq j\le R$ such that $(m_i,n_i)$ and $(m_j,n_j)$  
are not linearly dependent, then  the map  
$\T^1\dasharrow\T^2$, $ t\mapsto (f(t),g(t))$ is birational  
for generic $(v_i,v_j)\in\K^2.$  
\end{proposition}  
  
In particular when the vectors $(m_1,n_1),\dots, (m_R,n_R)$ are not collinear, the  
pa\-ra\-me\-trization  defined by the functions in~(\ref{kk}) is birational for generic  
$v_1,\dots, v_R$.  
The case when all of the $(m_i,n_i)$'s are collinear is easy to handle:  
we then set  
$$  
(m_i,n_i)= k_i(m,n)  
$$  
for some  $m,n\in \Z$ with $\gcd(m,n)=1$ and $k_i\in \Z\setminus\{0\}$.  
Setting $h(t):=\prod_{i=1}^R(t-v_i)^{k_i}$  
the functions in~(\ref{kk}) can be written as  
$$  
f(t)= \alpha \,h(t)^m \quad, \quad g(t)=\beta \,h(t)^n ,  
$$  
and hence the degree of the map $\rho:t\mapsto (f(t),g(t)) $ equals to the degree of $t\mapsto h(t)$ and so $\deg(\rho)=\eta(h)$.  
In particular, for generic $v_1,\dots, v_R\in \K$  
\begin{equation}\label{display:1}  
\deg(\rho)=\max\Bigg\{ \sum_{i:\ k_i>0}k_i , -\hspace{-3mm}\sum_{i:\ k_i<0}k_i \Bigg\}.  
\end{equation}

\begin{lemma}\label{t-v}  
Let $p\in \K(t)\setminus\K$, $q\in \K(t)^\times$ and $c\in \Z\setminus\{0\}$,  
then the map  
$$  
\T^1\dasharrow \T^2\quad , \quad  
t\mapsto(p(t),(t-v)^c q(t))  
$$  
is birational for generic $v\in \K$.  
\end{lemma}  
  
\begin{proof}  
Let $x_0\in \T^1$ generic, then the number of different solutions to $p(t)=x_0$  
equals $\eta(p)$  
and we denote these solutions $t_1,\dots, t_{\eta(p)}$. As $x_0$ is generic, we will have  
$q(t_i),q(t_j)\ne 0$.  
Suppose that for some $i\ne j$ and almost all $v\in \K$ we have  
$  
(t_i-v)^c q(t_i)= (t_j-v)^c q(t_j)  
$  
or equivalently  
$$  
\left(\frac{t_j-v}{t_i-v}\right)^c = \frac{q(t_j)}{q(t_i)}  
$$  
for almost all $v\in \K$.  
This implies that the rational function $v\mapsto (t_j-v)/(t_i-v)$ is constant and so $t_i=t_j$, a contradiction.  
We deduce then that for each $i\ne j$  
the set  
$$  
\{v\in \K: (t_i-v)^c q(t_i)= (t_j-v)^c q(t_j)\}  
$$  
is finite and  
hence  
 for $v$ outside all of these finite sets,  
the rational function $(t-v)^cq(t)$ separates the points  
$t_1,\dots, t_{\eta(p)}$.  
Let $C\subset \T^2$ be the Zariski closure of the image of  
the map $t\mapsto (p(t),(t-v)^c q(t))$ and $E_C\in \K[x^{\pm1},y^{\pm1}]$ its defining equation, the previous considerations show that  
$$  
\deg_y(E_C)=\Card(C\cap V(x-x_0))= \eta(p)  
$$  
for generic $x_0\in \T^1$.  
But on  the other hand,  
$\eta(p)\ge 1$ because $p$ is not constant and  
Lemma~\ref{gradoparcial} shows that  
$\deg_y(E_C)=\deg(\rho)^{-1}\eta(p)$,  
which concludes the proof.  
\end{proof}

\begin{Proof}{Proof of proposition~\ref{grado generico}.}  
Setting  $\Delta:=\det\begin{pmatrix} m_i &m_j\\ n_i&n_j \end{pmatrix}\ne 0$  
we have  
$$f^{n_j}g^{-m_j} = (t-v_i)^\Delta  F(t) \quad ,\quad  
f^{-n_i}g^{m_i} =(t-v_j)^\Delta G (t)  
$$  
for some rational functions $F,G $ not depending on the parameters $v_i,v_j$.  
The functions  
$$  
p(t):=(t-v_i)^\Delta  F(t)  
\quad ,\quad  
q(t):=G(t)  
$$  
are in the hypothesis of Lemma~\ref{t-v}, because  
$p\notin \K$ for all $v_i\in \K$ except a finite number of exceptions. For those $v_i$'s, this lemma shows that the map  
$$  
t\mapsto  
((t-v_i)^\Delta  F(t),  
(t-v_2)^\Delta G(t) )=  
(f^{n_j}(t)g^{-m_j}(t),  
f^{-n_i}(t)g^{m_i}(t))  
$$  
is birational except for a finite number of $v_j$'s, and {\it a fortiori} the same is true for $t\mapsto (f(t),g(t))$.  
\end{Proof}  
  
\medskip  
The following result shows that for a given parametrization, the  
Newton polygon can be computed by using partial factorizations. This can simplify  
the application of Theorem~\ref{mt} in concrete situations.  
  
\begin{lemma}\label{practical}  
Let $\rho=(f,g)\in (\K(t)^\times)^2$ and write  
\begin{equation}\label{factorizations}f  
(t)= \alpha \prod_{p\in \cP} p(t)^{d_p}  
\quad,\quad  
g(t)= \beta \prod_{p\in \cP} p(t)^{e_p}  
\end{equation}  
for a finite set $\cP\subset \K[t]$ of pairwise coprime polynomials, $d_p,e_p\in \N$  
and $\alpha,\beta\in \K^\times$, then  
$$  
P\big((\ord_v(\rho))_{v\in \P^1}\big) = P\big(\deg(p)(d_p, e_p))_{p\in \cP}, (\deg(f),\deg(g))\big).  
$$  
\end{lemma}  
  
\begin{proof}  
Let $p\in \cP$ and $v\in \K$ be a root of $p$ of multiplicity $m_v\ge 1$, then  
$$  
\ord_v(\rho)=m_v(d_p,e_p) \in (\R_{>0})(d_p,e_p).  
$$  
Hence all of the multiplicities of $\rho$ at roots of a given $p\in \cP$ lie in the same direction and so in the construction of  
$P\big((\ord_v(\rho))_{v\in \P^1}\big)$  they concatenate together  
into the single vector  
$$  
\sum_{v\in \K: \ p(v)=0}\ord_v(\rho)=\deg(p)(d_p,e_p),  
$$  
while on the other hand $\ord_\infty(\rho)= (\deg(f),\deg(g)) $.  
The fact that the polynomials in $\cP$ are pairwise coprime ensures that  
each non-zero  vector  
in the family $(\ord_v(\rho))_{v\in \P^1}$ appears exactly one time as a term of  
$ \deg(p)(d_p,e_p)$ for $p\in \cP$ or as  $(\deg(f),\deg(g)) $, which concludes the proof.  
\end{proof}

\medskip  
\begin{Proof}{Proof of Corollaries~\ref{genericLaurent},  
\ref{genericrationalequaldenom} and~\ref{genericrational}.}  
Corollary~\ref{genericLaurent} is contained both in  
Corollary~\ref{genericrationalequaldenom} and  
Corollary~\ref{genericrational},  so it suffices to prove these two  
statements.  
  
\smallskip  
Set  
$  
\wt p(t):=\alpha_d+\alpha_{d+1}t + \cdots+\alpha_Dt^{D-d} $ and  
$\wt q(t):=\beta_e+\beta_1t +\cdots+\beta_Et^{E-e} $  
so that the parametrization  $  
\rho=(f,g)$  
in Corollary~\ref{genericrationalequaldenom} factorizes as  
\begin{equation}\label{display:2}  
f(t)=t^d \, \wt p(t) \, r(t)^{-1} \quad ,\quad g(t)= t^e \, \wt q(t) \,  
r(t)^{-1} .  
\end{equation}  
In case $p,q,r$ are pairwise coprime, Theorem~\ref{mt} together with Lemma~\ref{practical} above readily imply that  
\begin{align}\label{newton_genericrationalequaldenom}  
\newton\big({\rho^*(\T^1)}\big)  
= P\big(& (\ord_0(f),\ord_0(f,g))  
, (\deg(\wt p), 0), (0,\deg(\wt q)), \nonumber\\  
& (-\deg(r), -\deg(r)), (\deg(f),\deg(g))\big) \nonumber \\[1mm]  
= P\big(&(d,e)  
,(D-d,0),(0,E-e),  
(-F,-F), (F-D,F-E)\big) .  
\end{align}  
Reciprocally, we will show that the  
equality~(\ref{newton_genericrationalequaldenom}) above  
holds only if $p,q,r$ are pairwise coprime. To see this, note first that  
$\ord_0(\rho)=(d,e)$ and $\ord_\infty(\rho)=(F-D,F-E)$ because  
$\alpha_d,\alpha_D, \beta_e,\beta_E\ne 0$.  
For rational functions of the form~(\ref{display:2}),  
the only factor which might contribute to the direction $(1,0)$  
is the polynomial $\wt p$ and it does contribute with the vector $(D-d,0)$  
if and only if  
$\wt p$ is coprime with both  $r$ and $\wt q$.  
Similarly, the presence of the vector $(0,E-e)$ implies that $\wt q$ is coprime with both  $r$ and $\wt p$ and so we conclude that   $p,q,r$ are pairwise coprime.  
  
\smallskip  
For the second part of the statement,  
assume that $(D-d,0),(0,E-e), (d,e),(F,F)$ are not collinear and  
for $v_1, v_2,v_3\in \K$  
consider the parametrization $\rho$ associated to the specialization  
$$  
p(t):= t^d(t-v_2)^{D-d} \quad , \quad  
q(t):=t^e(t-v_3)^{E-e} \quad , \quad  
r(t):=(t-v_4)^F.  
$$  
In case $(D-d,0),(0,E-e), (d,e)$ are not collinear,  
Proposition~\ref{grado generico} readily implies that  
$\rho$ is birational for generic $v_1,v_2,v_3$.  
If this is not the case, we have that $(d,e)$ and at least one of the vectors  
$(D-d,0),(0,E-e), (d,e)$ are linearly independent; we note $1\le i\le 3$ the corresponding index. Letting  
$j\ne i$ another index,  
we apply Proposition~\ref{grado generico} to  
the parametrization $\rho(t+v_j)$ and we also obtain that  
$\rho $ is birational for generic  
$v_1,v_2,v_3$.  
This readily implies that $\rho$ is birational for a generic choice of  
$p,q,r$ because  $\deg(\rho)=1$ is an open condition.  
  
\smallskip  
The proof of  Corollary~\ref{genericrational} is analogous and we leave it to the interested reader.  
\end{Proof}  
  
\medskip  
\section{The Variety of Rational Plane  
Curves with Given Newton Polygon} \label{variedad}  
  
This section is devoted to the proof of Theorem~\ref{M_Q} and more generally  
to the study of the variety $M_Q$.  
In particular, we show that every non-degenerate polygon is the Newton polygon of a rational plane curve.  
Similarly, we also determine which polygons  
can be the Newton polygon of a curve parameterized  
by polynomials or by Laurent polynomials.

\bigskip  
Throughout this section we denote by $Q$  
an arbitrary non-degenerate convex lattice polygon of $\R^2$.  
Setting $\tJ:=\Card(Q\cap\Z^2)-1$,  
we identify the space of Laurent polynomials with support contained in $Q$  
with  
$\K^{\tJ+1}$ and the set of effective Weil divisors of $\T^2$  
with Newton polygon contained in $Q$, with $\P^\tJ$.  
As defined in the introduction, the set  
$M_Q^\circ\subset \P^\tJ$  
consists in the effective divisors of  
$\T^2$  
of the form $\delta[C]$ for a rational curve $C$ and $\delta\ge1$ such that  
$\delta\newton(C)=Q$.  
We denote by $M_Q\subset\P^\tJ$  the Zariski  
closure of this set.  
  
\smallskip  
In the sequel we construct the space of rational parametrizations corresponding to  
the divisors in~$M_Q^\circ$.  
Let $(m_1,n_1), \dots, (m_r,n_r)\in \Z^2\setminus\{(0,0)\}$ be  
the primitive inner normal vectors of $Q$  
and  
$\ell_1,\ldots,\ell_r\ge 1$ the lattice length of the corresponding  
edge.  
Given a  point  $p_i=(p_{i,0}:\cdots:p_{i,\ell_i})  
\in \P^{\ell_i}$,  we identify it  
with the associated monic  polynomial $p_i(t)\in \K[t]$  
of degree $\le \ell_i$. In concrete terms,  
$$  
p_i(t):=p_{i,k}^{-1}\, p_{i,0}+\cdots +p_{i,k}^{-1}\, p_{i,k-1}t^{k-1}+t^k  
$$  
for the maximal index $0\le k\le \ell_i$ such that  
$p_{i,k}\ne 0$.  
We then consider the dense open subset  
$$  
U_Q  
\subset \T^2\times \prod_{i=1}^r \P^{\ell_i}  
$$  
of points  $  
(\alpha,\beta,p_1,\dots ,p_r) $ such that the homogeneous  
associated polynomials  
$$  
p_i^\h(t_0,t_1):=t_0^{\ell_i}\, p_i(t_0^{-1}t_1)\in \K[t_0,t_1]  
\quad \mbox{ for  } 1\le i\le r  
$$  
are pairwise coprime.  
To each $\bfu=(\alpha,\beta,p_1,\dots ,p_r)  
\in U_Q$ we associate  a parametrization $\rho_\bfu=(f_\bfu,g_\bfu)\in \K(t)^2$  
defined by  
\begin{equation}\label{eq:1}  
f_\bfu(t):=\alpha\,\prod_{i=1}^r p_i(t)^{m_i}  
\quad ,\quad  
g_\bfu( t):=\beta\,\prod_{i=1}^r p_i(t)^{n_i}.  
\end{equation}  
  
\begin{lemma}\label{parameters}  
The map  
$  
U_Q\to \K(t)^2 , \bfu\mapsto \rho_\bfu$  
is a one-to-one correspondence between $U_Q$ and the set of rational maps  
$\rho:\T^1\dashrightarrow \T^2$ such that $\newton(\rho^*(\T^1))=Q$.  
\end{lemma}  
  
\begin{proof}  
According to Theorem \ref{mt},  
a rational map $\rho:\T^1\dashrightarrow \T^2$  
corresponds to a member of $M_Q^\circ$  
if and only if for all $v\in \P^1$ there is and index $i$ between $1$ and $r$  
such that  
$  
\ord_v(\rho) \in (\R_{\ge 0})(m_i,n_i)  
$  
and  
$$  
{\sum_{v}} \ell(\ord_v(\rho))=\ell_i  
\quad \mbox{ for } 1\le i\le r ,  
$$  
the sum being over $v\in \P^1$ such that $\ord_v(\rho) \in (\R_{>0})(m_i,n_i)$.  
Every such map can be written as  
$\rho_\bfu$ for $\bfu=(\alpha,\beta,p_1,\dots, p_r)\in U_Q$, with  
$(\alpha,\beta)=\init_\infty(\rho)$ and  
$$  
p_i(t)=\prod_{v\ne \infty}  
(t-v)^{\ell(\ord_v(\rho))} \quad \mbox{ for } 1\le i\le r,  
$$  
this time the index variable runs over $v\in \K$  
such that $\ord_v(\rho) \in (\R_{>0})(m_i,n_i)$.  
Reciprocally,  
the map $\rho_\bfu$ corresponding to a point  
 $\bfu=(\alpha,\beta,p_1,\dots, p_r)\in U_Q$ verifies the above conditions:  
we have that $\ord_v(\rho_\bfu)\ne (0,0)$ if and only if $v\in V(p_i^\h)$ for  
$1\le i\le r$ and in that case,  
$\ord_v(\rho)=\ord_v(p_i^\h)(m_i,n_i) \in (\R_{>0})(m_i,n_i)$.  
Moreover  
$$  
\sum_{v\in V(p_i^\h)}  \ell(\ord_v(\rho_\bfu))  
=  
\sum_{v\in V(p_i^\h)}  
\ord_v(p_i^\h) =\ell_i  
$$  
and so $\newton(\rho_\bfu^*(\T^1))=Q$.

It only remains to show that the correspondence $\bfu\mapsto \rho_\bfu$ is  
injective:  
let $\bfu=(\alpha,\beta,p_1,\dots, p_r)  
,\bfu'=(\alpha',\beta',p_1',\dots, p_r')\in U_Q$  
such that $\rho_\bfu=\rho_{\bfu'}$.  
For $i\ne j$ set  
$\Delta_{i,j}:=\det\begin{pmatrix} m_i &m_j\\ n_i&n_j \end{pmatrix}\ne 0$,  
so that for $1\le i\le r$  
$$  
\alpha^{n_i}\beta^{-m_i}  
\prod_{j\ne i}p_j^{\Delta_{i,j}}  
=  
f_\bfu^{n_i}g_\bfu^{-m_i}  
=f_{\bfu'}^{n_i}g_{\bfu'}^{-m_i}  
= (\alpha')^{n_i}(\beta')^{-m_i}  
\prod_{j\ne i}(p_j')^{\Delta_{i,j}}.  
$$  
This shows that  
$p_i(t)$ and $p_j'(t)$ are coprime for $i\ne j$ and so  
$p_i(t)=p_i'(t)$ for all $i$.  
Moreover  
$  
(\alpha,\beta)=\init_\infty(\rho_\bfu)=\init_\infty(\rho_{\bfu'})=(\alpha',\beta')  
$  
and we conclude that $\bfu=\bfu'$, as desired.  
\end{proof}  
  
Proposition~\ref{grado generico} together with this Lemma readily imply  
that the parametrization $\rho_\bfu$ is birational for generic $\bfu\in U_Q$.  
  
\begin{corollary}\label{degreeQ}  
Let $Q$ be a non-degenerate convex lattice polygon,  
then  $\deg(\rho_\bfu)=1$ for $\bfu$ in a dense open subset of $U_Q$.  
\end{corollary}

\medskip  
Different parametrizations might define the same push-forward cycle.  
To avoid most of this redundancy, we introduce an equivalence relation by  
agreeing that  
two pa\-ra\-me\-tri\-zations  $\rho, \rho':\T^1\dashrightarrow \T^2$ are  
{\it equivalent} if there exists a birational automorphism  
$\mu$ of $\T^1$ such that  
$\rho'=\rho\circ \mu$. The  
birational automorphisms  
of $\T^1$ are the M\"obius transformations, that is the  
maps of the form  
$\T^1\dashrightarrow \T^1,  
t\mapsto\frac{\alpha t+\beta}{\gamma t+\delta}$ for some  
$\alpha,\beta,\gamma,\delta\in\K$ such that $\alpha\delta-\beta\gamma\ne 0$.  
We then set  
$$  
P_Q:=U_Q/\sim  
$$  
for the space of parametrizations of effective divisors in $M_Q^\circ$ modulo equivalence.  
The assumption that $Q$ is non-degenerate implies that it  
has at least three different edges, that is $r\ge 3$.  
Given any  three different  
points of $\P^1$, there is a unique M\"obius transformation  
which brings them to $0,1,\infty$ respectively.  
We can thus fix a system of representatives of $P_Q$ as the set of points  
 $\bfu=(\alpha,\beta,p_1,\dots ,p_r)  
\in U_Q$ such that the polynomials $p_1(t)$, $p_2(t)$ and $p_3(t)$  
have respectively $0$, $1$ and $\infty$ as a root.  
This identifies  
$P_Q$ with a subset of $U_Q$. We can also verify that $P_Q$  
is linearly isomorphic to  
a dense open subset of $$\T^2\times\prod_{i=1}^3\P^{\ell_i-1}\times\prod_{i=4}^r\P^{\ell_i}.$$

\smallskip  
It is well-known that the equation of the  
cycle $\rho^*(\T^1)$ can be written  
in terms of the Sylvester resultant.  
Set  
$$  
D := \sum_{i: \ m_i>0} \ell_im_i =- \hspace{-2mm}\sum_{i: \ m_i<0} \ell_im_i  
\quad ,\quad  
E:=\sum_{i: \ n_i>0} \ell_in_i =- \hspace{-2mm} \sum_{i: \ n_i<0} \ell_in_i ,  
$$  
these equalities being a consequence of the balancing condition  
$\sum_{i=1}^r\ell_i(m_i,n_i)=(0,0)$.  
We then consider the map $ U_Q \to \K[x^{\pm 1},y^{\pm1}]$ defined by  
\begin{equation} \label{eq:2}  
\rho_\bfu=(f,g) \mapsto  
x^ay^b\Res_{D,E} (f_\den(t)\, x-f_\num(t),g_\den(t)\, y-g_\num(t);t)  
\end{equation}  
where $\Res_{D,E}$ stands for the Sylvester resultant two univariate polynomials  
of degree $D$ and $E$ respectively and  $a,b\in \Z$ are maximal  
so that $Q-(a,b)\in (\R_{\ge0})^2$.  
Note that for $\bfu \in U_Q$ we have $\eta(f_\bfu)= D$ and $\eta(g_\bfu)=E$ and  
so $$  
\deg_t(f_\den(t)\, x-f_\num(t))= D \quad , \quad  
\deg_t(g_\den(t)\, y-g_\num(t))=E,  
$$  
hence the map~(\ref{eq:2}) is well-defined on the whole of $U_Q$.  
  
\begin{proposition}\label{rational map}  
With the above notation, Formula~(\ref{eq:2}) induces  
a surjective regular map  
$\Xi:P_Q\to M_Q^\circ$.  
The fiber of $\Xi$ at $\delta[C]\in M_Q^\circ$  
is an irreducible variety of dimension ${2\delta-2}$.  
\end{proposition}  
  
\begin{proof}  
Recall that for univariate polynomials $p,q$ of degree bounded by  
$D$ and $E$ respectively,  
the  
Sylvester resultant $\Res_{D,E}(p,q)$ vanishes if and only if  
$\gcd(p,q)\neq1$, or $\deg(p)<D$ and $\deg(q)<E$.  
Put  
$$  
R_\bfu (x,y):=\Res_{D,E}( f_\den(t)\, x-f_\num(t),  
g_\den(t)\, y-g_\num(t);t)\in \K[x,y].  
$$  
In our setting we have that  
$\deg_t(f_\den(t)x-f_\num(t))=D$ and $\deg_t(g_\den(t)y-g_\num(t))=E$  
and they are coprime as polynomials in $\K(x,y)[t]$.  
This implies that $R_\bfu\ne0$  
 and moreover $R_\bfu(0,y),R_\bfu (x,0) \ne 0$.  
We can conclude then that $R_\bfu $ is  
is a defining equation for the divisor $\delta[C]$, which is a  
Taylor polynomial neither divisible by $x$ nor by~$y$.  
Hence $\newton(R_\bfu)= Q-(a,b)$ and  
 this is equivalent to  $x^ay^b R_\bfu\in M^\circ_Q$.  
  
This shows that  
Formula~(\ref{eq:2}) defines  
a regular surjective map from $U_Q$ onto $M_Q^\circ$.  
This map is well-defined in the quotient space $P_Q$ because  
it only depends on the push-forward cycle $\rho^*(\T^1)$ and moreover, it  
is regular  
on this quotient space because $P_Q$ can be regarded as a subset of $U_Q$.

\medskip  
For a given divisor $\delta[C]\in M_Q^\circ$ set  
$$  
\Omega_\bfu:=\{\bfu :  
\newton(\rho_\bfu^*(\T^1))=\delta[C]\}\subset U_Q.  
$$  
By L\"uroth theorem~\cite[\S~V.7, pp.~149-151]{Wal50} there exists a  
birational map $\psi:\T^1\dashrightarrow C$, and so $\Omega_\bfu$ is isomorphic to the  
space of birational endomorphism of $\T^1$ of degree $\delta$ through the  
map  
$$  
\End(\T^1;\delta)\to \Omega_\bfu \quad ,\quad  
h\mapsto \psi\circ h,  
$$  
with inverse $\rho\mapsto\psi^{-1}\circ \rho$.  
The group $\Aut(\T^1)$ of M\"obius transformations  
 acts on $\End(\T^1;\delta)$ by right composition and the fiber $\Xi^{-1}(\delta[C])\subset P_Q$ identifies with the set-theoretical quotient  
 $\End(\T^1;\delta)/\Aut(\T^1)$.

In case $\delta=1$ we have that  
$\End(\T^1;1)=\Aut(\T^1)$  
so in this case the quotient is a point.  
For $\delta\ge 2$, we affirm that the orbit of a generic element  
$h\in \End(\T^1;\delta)$ is isomorphic to $\Aut(\T^1)$:  
if we have $h=h\circ \mu$ for $\mu\in \Aut(\T^1)$, the M\"obius transformation  
$\mu$ must fix the set of zeros and poles of $h$.  
This set has at least four generic points of $\T^1$ because $h$ is generic and $\delta\ge 2$, and so  
$\mu=\Id_\T^1$.  
Hence the fiber of  
$$  
\pi:\End(\T^1;\delta) \to \End(\T^1;\delta)/\Aut(\T^1)  
$$  
over a generic point $w$ has dimension 3, and so by the theorem of dimension of fibers  
$$  
\dim\left(\End(\T^1;\delta)/\Aut(\T^1)\right)  
=\dim\left(\End(\T^1;\delta)\right)-\dim\pi^{-1}(w)  
=(2\delta+1)-3,  
$$  
which concludes the proof.  
\end{proof}  
  
From the proof above, we see that the fiber $\Xi^{-1}(\delta[C])$ is in fact isomorphic  
to the quotient space $\End(\T^1;\delta)/\Aut(\T^1)$.

\begin{corollary}\label{delta}  
The set of reduced divisors in  $M_Q^\circ$  
contains a dense open subset.  
In particular,  
$\Xi$ is a bijection between some dense open subsets of $P_Q$ and $M_Q^\circ$.  
\end{corollary}  
  
\medskip  
This is a direct consequence of  
Corollary~\ref{degreeQ} and the fact that the map  
$\Xi:P_Q\to M_Q^\circ$ is  dominant.  
We are now in position to prove  Theorem~\ref{M_Q}.

\medskip  
\begin{Proof}{Proof of Theorem~\ref{M_Q}.}  
Set $\ell:=\sum_{i=1}^r\ell_i= \Card(\partial Q \cap \Z^2)$.  
By Corollary~\ref{delta},  
we have that $\Xi$ is a bijection between some dense open subsets of $P_Q$ and $M_Q^\circ$ and so the  
theorem of dimension of fibers implies that  
$$  
\dim(M_Q)=\dim(P_Q) =  
2+\sum_{i=1}^3(\ell_i-1) + \sum_{i=4}^r\ell_i= \ell-1.  
$$  
In particular  
$M_Q$ is unirational, because $P_Q$ is birational to  
$\T^2\times\prod_{i=1}^3\P^{\ell_i-1}\times\prod_{i=4}^r\P^{\ell_i}$ and {\it a fortiori} to $\A^{\ell-1}$.  
  
\smallskip  
In case $\Char(\K)=0$,  
the field extension  
$  
\Xi^*(\K(M_Q))\subset \K(\A^{\ell-1})  
$  
is separable  
and so  
its degree equals the cardinality of a generic fiber~\cite[\S~II.6.3, pp.~142-145]{Sha94}.  
Hence  
$\Xi^*(\K(M_Q))= \K(\A^{\ell-1})$  
and in particular  $M_Q$ is rational.  
\end{Proof}  
  
\medskip  
Set $c(Q)$ for the {\it content} of $Q$, defined as  
the maximal integer $c\ge1$ such that  
$c^{-1} Q$ is a lattice polygon modulo a translation.  
The polygon $Q$ is {\it non-contractible} whenever $c(Q)=1$.

For $\delta[C]\in M_Q^\circ$ we have  
$$  
Q= \newton(\delta[C])= \delta\newton(C)  
$$  
and so $\delta|c(Q)$ because $\newton(C)$ is a lattice polygon.  
Hence the possible multiplicities of a divisor in $M_Q^\circ$ are exactly the divisors of $c(Q)$.  
In particular, $\Xi:P_Q \to M_Q^\circ$ is a bijection  
if and only if $Q$ is non-contractible.  
  
\smallskip  
Alternatively, the  
content of $Q$  
coincides with the greatest common divisor of the lattice length  
of its edges.  
Hence in the setting of Corollary~\ref{degreeQ},  for $\bfu\in U_Q$ we have  
$\deg(\rho_\bfu)|\gcd(\ell_1, \dots, \ell_r)$.  
In particular,  $\rho_\bfu$ is birational for all $\bfu\in U_Q$  
if and only if $\gcd(\ell_1, \dots, \ell_r)=1$.  
  
\medskip  
The following example shows that in positive characteristic,  
the map $\Xi$ can be generically ramified.

\begin{example}\label{example:5}  
Consider the polygon $Q$ in Figure~\ref{fig:4} below.  
%\vspace{-4mm}  
\begin{figure}[htbp]
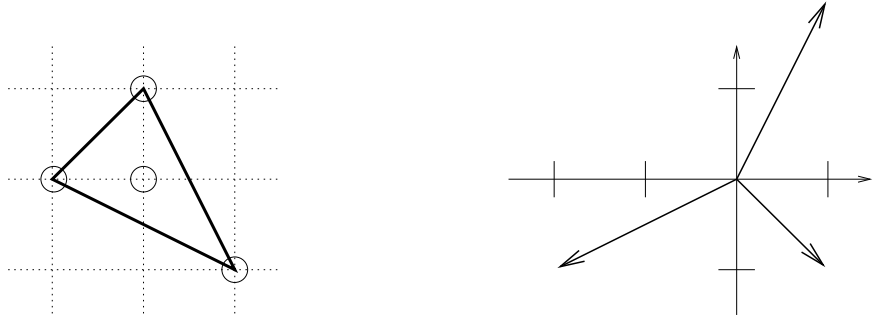
  
\input fano1.pstex_t  
\vspace{-3mm}\caption{  
A reflexive polygon and its primitive inner normals}\label{fig:4}  
\end{figure}  
Its primitive inner normals are the vectors $(1,2),(-2,-1), (1,-1)$.  
This is a non-contractible polygon and so all of the divisors in $M_Q^\circ$ are reduced.  
A parametrization of a rational plane curve  with Newton polygon $Q$  
is equivalent modulo a M\"obius transformation  
to one of the form $\rho=(f,g)$ with  
$$  
f=\alpha t (t-1) \quad ,\quad g=\beta \frac{t^2}{t-1}  
$$  
for some $\alpha,\beta\in \K^\times$.  
Hence $P_Q=\T^2$ and we have that  
$$  
\Res_{2,2}\big( x-\alpha t(t-1), (t-1)y-\beta t^2; t \big) =  
\beta^2 x^2+ 3\alpha\beta xy -\alpha x y^2 + \alpha^2\beta  y.  
$$  
The map $\Xi$ then explicites as  
$$  
\Xi:\T^2\to \P^3 \quad ,\quad  
(\alpha,\beta)\mapsto (\beta^2: 3\alpha\beta  : -\alpha: \alpha^2\beta).  
$$  
Let $z_{x^2}, z_{xy},z_{xy^2}, z_{y}$ denote the homogeneous coordinates of $\P^3$, corresponding to  
the monomials whose exponents are the lattice points of $Q$.  
Then we have that $M_Q$ is a surface in $\P^3$  
with homogeneous defining  
equation  
$$  
z_{xy}^3-27 z_{x^2}\, z_{xy^2}\, z_{y} =0 .  
$$  
If $\Char(\K)\ne 3$, this is a  smooth toric  
surface of degree 3.  
On the other hand, for $\Char(\K)=3$ the map  
$\Xi$ is  generically ramified of degree~$3$.  
The defining equation of $M_Q$ is  
$  
z_{xy}=0  
$,  
hence it is a rational surface, but of degree~$1$.  
\end{example}

In all the examples computed so far, the variety $M_Q$ is always rational independently of the characteristic of the base field~$\K$. It would be interesting  
to determine if this holds in general.  
This does not follow immediately from Lemma~\ref{rational map},  
as in positive characteristic there are examples of  
non-rational surfaces whose field of fractions admits $\K[\A^2]$ as  
a purely inseparable extension~\cite{Shi74}.

\begin{example}\label{example:6}  
Consider the polygon $\wt{Q}$ in the figure below:  
\begin{figure}[htbp]
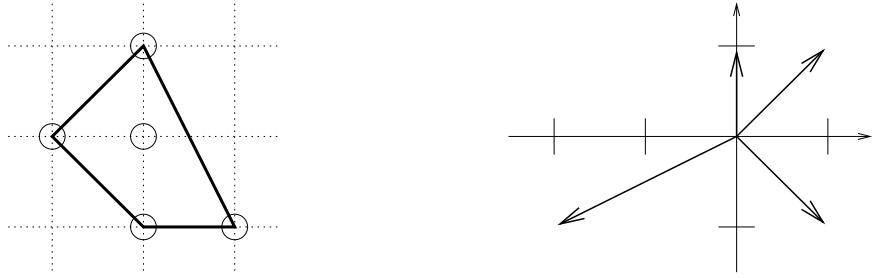
  
\input fano2.pstex_t  
\vspace{-3mm}\caption{  
A second reflexive polygon and its primitive inner normals}\label{fig:5}  
\end{figure}  
  
Its inner normals are the vectors $(1,1),(0,1),(-2,-1),(1,-1)$.  
Hence a parame\-tri\-zation of a rational plane curve with this Newton polygon  
is equivalent modulo a M\"obius transformation  
to one of the form $\rho=(f,g)$ with  
\begin{equation}\label{eq:3}  
f=\alpha t (t-1) \quad ,\quad g=\beta \frac{t(t-v)}{t-1}  
\end{equation}  
for some $\alpha,\beta\in \K^\times$ and $v\in\K\setminus\{0,1\}$.  
Let  
$z_{x^2}, z_{xy}, z_{xy^2}, z_x, z_{y}$ denote the homogeneous coordinates of $\P^3$, corresponding to  
the monomials whose exponents are the lattice points in $\wt{Q}$.  
We have that $P_{\wt{Q}}=\T^2\times (\K\setminus\{0,1\})$. By computing the corresponding resultant,  
 the map $\Xi$ expresses as  
\begin{align*}  
\Xi:&\T^2\times (\K\setminus\{0,1\}) \to \P^4 \\[1mm]  
&(\alpha,\beta,v)\mapsto  
\big(\beta^2: \alpha\beta (3-2v):  -\alpha: \alpha\beta^2v(v-1):  
\alpha^2 \beta  (1-v)\big).  
\end{align*}  
In characteristic $0$, the variety $M_{\wt{Q}}$ is a hypersurface of $\P^4$ of degree $5$ with  
homogeneous defining equation  
$$  
-z_{y}\,z_{x^2}\,z_{xy}^3+z_{x}\,z_{xy}^4-8\,z_{x}^2\,z_{xy^2}\,z_{xy}^2+16\,z_{x}^3\,z_{xy^2}^2-27\,z_{y}^2\,z_{xy^2}\,z_{x^2}^2+36\,z_{x}\,z_{xy^2}\,z_{y}\,z_{x^2}\,z_{xy}=0.  
$$  
Its singular locus contains the surface defined by the equations  
$$  
27\, z_y\, z_{x^2}\, z_{xy^2} +2z_{x,y}^3=0  
\quad ,\quad  
12\, z_x\, z_{xy^2}+z_{xy}^2  
$$  
together with two lines and a conic.  
\end{example}  
  
\smallskip  
For a given polygon $Q$, the map $\Xi$ is regular in principle only on $U_Q$. It might be interesting  
to study what happens when we get to the border of $U_Q$ and more generally, to make explicit the  
Zariski closure of $M_Q$ in geometric or combinatorial terms.  
\par  
In this direction, the behavior of $\Xi$ when a root of  $p_i$ concides with one of $p_j$ in~(\ref{eq:1})  
for some $i\ne j$ has a nice combinatorial description.  
For instance, if we set $v=0$ in~(\ref{eq:3}), we obtain the  
parametrization in Example~\ref{example:5}. This gives a natural inclusion  
$$  
M_Q\hookrightarrow M_{\wt Q}  
$$  
which corresponds to a shrinking of the polygon $\wt Q$  
into the polygon $Q$ from Example~\ref{example:5},  
as shown in Figure~\ref{fig:6}.  
  
\vspace{-2mm}  
\begin{figure}[htbp]
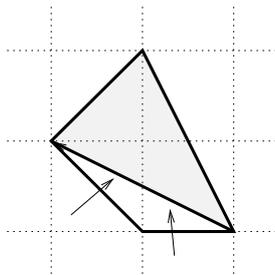
  
\input shrinking.pstex_t  
\vspace{-3mm}\caption{  
Polygon $\wt Q$ shrinks into polygon $Q$}\label{fig:6}  
\end{figure}  
  
This can be explained by the merging of the primitive inner normals $(1,1)$  
and $(0,1)$ into its sum $(1,2)$.  
These observations can be extended to other cases when  
a root of $p_i$ coincides with one of $p_j$ for  $i\ne j$,  
but does not seem to suffice to explain all of the points in  
$M_Q\setminus M_Q^\circ$.  
  
\medskip  
The following result follows immediately from  
Theorem~\ref{M_Q}.  
  
\begin{corollary}  
Let $Q\subset\R^2$ be a non degenerate convex lattice polygon, then  
$M_Q=\P^\tJ$ if and only if $Q$ has no interior points, and  
$M_Q$ is a hypersurface if and only if $Q$ has exactly one interior point.  
\end{corollary}

The first part of this corollary is also a consequence of Baker's formula  
for the geometric  
genus of a curve in a toric surface~\cite{Bak1893,Kho78,BS07}, as this formula implies that the geometric genus of a  
 generic plane curve with Newton polygon  
$Q$ equals  $\Card(Q^\circ\cap \Z^2)$.  
  
Convex lattice polygons without interior points can be easily classified.  
There is an infinite number of them, and  
modulo lattice equivalence they are either  
of the form $\Conv\big((0,0),(0,2),(2,0)\big)$ or  
$$  
\Conv\big((0,0),(0,1),(a,1),(b,0)\big)  
$$  
with $ a\ge 0$ and $ b\ge 1 $ such that $a\le b$.  
  
Polygons with exactly one interior point  
have nice properties and have been extensively studied.  
A polygon has exactly one interior point if and only if  
it is {\it reflexive} in the sense that  
its polar polygon is again a lattice polygon.  
Reflexive polygons are important because they  exactly  correspond  
to (non necessarily smooth) Fano toric surfaces.  
There are~$16$ reflexive polygons modulo  
lattice equivalence and pictures of them can be found for instance  
in~\cite{PV00}.  
We used three of these polygons in our examples~\ref{example:1},  
\ref{example:5} and~\ref{example:6}.  
  
\bigskip  
A further natural question in this context is whether for a rational plane curve  
with given  
polygon~$Q$, all of the monomials associated to the lattice points in $Q$ appear.  
In general, the answer is ``no'' even for the generic member of $M_Q^\circ$.  
In characteristic~$3$,  
Example~\ref{example:5} actually shows that an equation  
with that Newton polygon defines a rational curve if and only if the  
coefficient of the monomial $xy$ is zero.  
It would be interesting to determine whether one can produce a similar example in characteristic~$0$.  
  
Note that the fact that the equation of the generic  
member of $M_Q$  
does not depend on the monomial  
$x^ay^b$  for a given $(a,b)\in Q$  
is equivalent to say that $M_Q$ is contained in the standard hyperplane  
$V(z_{x^ay^b})$ of $\P^\tJ$.  
  
It would also be interesting to determine the degree  
of~$M_Q$ as a projective variety and its singular locus.  
  
\bigskip  
Finally we prove Theorem~\ref{polly} on the shape of the Newton polygon of  
plane curves parameterized by polynomials or by Laurent polynomials.  
  
\medskip  
\begin{Proof}{Proof of Theorem~\ref{polly}.}  
Let $Q$ be a non-degenerate convex lattice polygon such that  
$Q= \newton\big(\ov{\rho(\T^1)}\big)$ for  
 $\rho\in \K[t]^2$.  
Setting $\delta:=\deg(\rho)$ we have $\delta Q= \newton(\rho^*(\T^1))$ and this holds  
if and only if $\rho$ corresponds to a point  
 $\bfu=(\alpha,\beta,p_1,\dots, p_r)\in U_{\delta Q}$  
with $p_i(t)=1$ whenever $(m_i,n_i)\notin\N^2$.  
Let $i_0$ be any of such indexes. We have then $V(p_{i_0}^\h)=\{\infty\}$, and there  
is at most one of these indexes with this property,  because the polynomials  
$p_1^\h, \dots, p_r^\h$ are pairwise coprime.  
\par  
On the other hand, not all of the $(m_i,n_i)$'s  can lie in $\N^2$ because of the  
balancing condition and so there is  
exactly one such index. We conclude that $(\R_{>0})(m_{i_0},n_{i_0})$  
is the only inner normal direction  
of $Q$  
which does not lie in $(\R_{\ge 0})^2$.  
  
Conversely, assume that $Q$ has exactly  
one inner normal direction not lying in  $(\R_{\ge 0})^2$.  
Then all of the  
points  $\bfu=(\alpha,\beta,p_1,\dots, p_r)  
\in U_Q$ such that  
 $p_{i_0}(t)=1$  
for the index $i_0$ corresponding to that inner direction, give  
parametrizations $\rho_\bfu$ such that $\newton(\rho_\bfu^*(\T^1))=Q$, By Proposition~\ref{grado generico}  
the map $\rho_\bfu$ is birational for $\bfu $ generic within the considered  
points and hence  
$Q= \newton\big(\ov{\rho_\bfu(\T^1)}\big)$ for such $\bfu$'s.  
  
\medskip  
Similarly,  
$Q= \newton\big(\ov{\rho(\T^1)}\big)$ for  
 $\rho\in \K[t^{\pm1}]^2$  
if and only if $\rho$ corresponds to a point  
 $\bfu=(\alpha,\beta,p_1,\dots, p_r)\in U_{\delta Q}$  
such that $p_i(t)=t^{a(i)}$ for $a(i)\in \N$ whenever $(m_i,n_i)\notin\N^2$.  
For such an  index $i_0$ we have that $V(p_{i_0}^\h)\subset \{0, \infty\}$ and so there are at most  
{\it two} such indexes. As before, the balancing condition shows that there must be at least one such index.  
We conclude that in this case  
there are one or two  inner normal directions of $Q$  
which do not lie in $(\R_{\ge 0})^2$. The converse is proved with the same strategy of the previous case.  
\end{Proof}

\subsection{The degenerate case}\label{degenerate}  
  
We consider separately the case when $Q$ is a lattice segment, because  
the corresponding results are slightly different.  
If this is the case, modulo a translation we have  
$$  
Q= \Conv\big((0,0), (ka, kb)\big) \subset\R^2  
$$  
for a primitive vector $(a,b)\in \Z^2$ and $k:=\ell(Q)\ge 1$.  
The general Laurent polynomial with Newton polygon $Q$ is of the form  
$ F(x,y)=p(x^ay^b)$ for a polynomial  
$p(s)=p_0+p_1s+\cdots+p_ks^k$ such that $p_0,p_k\ne0$.  
Hence  
$$  
V(F)= \bigcup_{\xi\in V(p)} V(x^ay^b-\xi).  
$$  
Each component curve $V(x^ay^b-\xi)$ is rational,  
 as it  
can be parameterized by a map  
$\T^1\to \T^2, t\mapsto(\xi^c t^b, \xi^dt^{-b}) $ for integers $c,d$ such that $ca+db=1$.  
Hence  
$F\in M_Q^\circ$ if and only if $p$ has only one root in $\T^1$,  
or equivalently if and only if  
$  
F(x,y)=\nu(x^ay^b-\xi)^k$  
for some $\nu,\xi\in \K^\times$.  
We deduce that  
$$  
M_Q^\circ=\Big\{ \Big({k\choose j} t^j : 0\le j\le k\Big): t\in \T^1\Big\}\subset\P^k,  
$$  
which shows that $M_Q$ is linearly isomorphic to the Veronese curve  
in case $\Char(\K)=0$ or $>k$, and linearly isomorphic to a projection of this Veronese curve, in any characteristic.  
Thus the following is the analog of Theorem~\ref{M_Q} in our setting.  
  
\begin{proposition}\label{M_Q_segment}  
Let $Q\subset \R^2$ be a lattice segment,  
then  
$M_Q$  
is a rational curve.  
\end{proposition}  
  
Note that this dimension~$1$ of $M_Q$  
can be interpreted as the number of points in the {\it relative} boundary  
of $Q$ minus $1$.  
  
\medskip  
On the other hand, the multiplicity of all of the  divisors in  $M_Q^\circ$ equals~$\ell(Q)$, as they are of the form $  
F(x,y)=\nu(x^ay^b-\xi)^k$  
for some $\nu,\xi\in \K^\times$.  
Hence the lattice segment~$Q$ realizes as the Newton polygon of a rational plane  
curve if and only if $\ell(Q)=1$.  
  
\bigskip  
\bibliographystyle{plain}

\begin{thebibliography}{Bak1893}  
\bibitem[Bak1893]{Bak1893}  
H.F. Baker,  
\newblock {\it Examples of applications of Newton's polygon to the theory of  
singular points of algebraic functions.\/}  
Trans. Cambridge Phil. Soc. {\bf 15} (1893) 403-450.  
  
\bibitem[BS07]{BS07}  
T. Beck, J. Schicho,  
{\it Parametrization of algebraic curves defined by sparse equations.\/}  
Appl. Algebra Engrg. Comm. Comput. {\bf 18} (2007) 127-150.  
  
\bibitem[Ber75]{Ber75}  
D.N. Bern\v{s}tein,  
\newblock {\it The number of roots of a system of equations (in Russian).\/}  
Funk. Anal. Priloz.~{\bf 9\/} (1975) 1-4;  
English translation in Functional Anal. Appl. {\bf 9} (1975) 183-185.  
  
\bibitem[DFS07]{DFS07}  
A. Dickenstein, E.-M. Feichtner, B. Sturmfels,  
\newblock {\it Tropical discriminants.\/}  
\newblock J. Amer. Math. Soc.~{\bf 20} (2007) 1111-1133.  
  
\bibitem[EKP10]{EKP10}  
I. Emiris, C. Konaxis, L.  Palios, 
\newblock{\it Computing the Newton polytope of  the implicit equation.\/}  
\newblock  Math. Comput. Sci, this volume.  
  
\bibitem[EK03]{EK03}  
I. Emiris, I. Kotsireas,  
\newblock{\it Implicitization with polynomial support optimized for sparseness.\/}  
\newblock  
% {\em Proc. Intern. Conf. Comput. Science \& Appl. 2003}, Montreal, Canada (Intern. Workshop Computer Graphics \& Geom. Modeling)  
Lecture Notes in Comput. Sci. {\bf 2669} (2003) 397-406, Springer.  
  
\bibitem[EK05]{EK05}  
I. Emiris, I. Kotsireas,  
\newblock{\it Implicitization exploiting sparseness.\/}  
\newblock  
%In R. Janardan, J. Smid, and D. Dutta (eds.)  
%{\em Geometric and Algorithmic Aspects of Computer-Aided Design and Manufacturing,}  
DIMACS Ser. Discrete Math. Theoret. Comput. Sci. {\bf 67} (2005) 281-298,  
Amer. Math. Soc..  
  
\bibitem[EKho07]{EKho07}  
A. Esterov, A. Khovanski\u\i,  
\newblock{\it Elimination theory and Newton polytopes.\/}  
\newblock  Funct. Anal. Other Math.~{\bf 2}  (2008),  no. 1, 45--71. 
  
  
\bibitem[Fra06]{Fra06}  
M. Franz, {\tt Convex: a Maple package for convex geometry, version 1.1.\/}  
\newblock Available at {\tt http://www-fourier.ujf-grenoble.fr/$\sim$franz/convex/}.  
  
  
\bibitem[Jel05]{Jel05}  
Z. Jelonek,  
\newblock{\it On the effective Nullstellensatz.\/}  
\newblock  Invent. Math.  {\bf 162}  (2005) 1-17.  
  
\bibitem[Kho78]{Kho78}  
A. Khovanski\u\i,  
\newblock{\it Newton polyhedra, and the genus of complete intersections.\/}  
\newblock   Funk. Anal. i Priloz. {\bf 12} (1978) 51-61;  
English translation in Functional Anal. Appl. {\bf 12} (1978) 38-46.  
  
\bibitem[McM04]{McM04}  
P. McMullen,  
\newblock{\it Mixed fibre polytopes.\/}  
\newblock Discrete Comput. Geom. {\bf 32} (2004) 521-532.  
  
%\bibitem[Per07]{Per07}  
%S. P\'erez-D\'iaz,  
%\newblock{\it Computation of the singularities of parametric plane curves.\/}  
%\newblock J. Symb. Comput. {\bf42} (2007) 835-857.  
  
\bibitem[PS04]{PS04}  
P. Philippon, M. Sombra,  
\newblock {\it Hauteur normalis{\'e}e des vari{\'e}t{\'e}s  
toriques projectives.\/}  
\newblock Journal of the Institute of Mathematics of Jussieu (2008), ~{\bf 7}:2:327--373. 
  
\bibitem[PS07a]{PS07a}  
P. Philippon, M. Sombra,  
\newblock{\it Une nouvelle majoration pour le nombre de solutions d'un  
syst\`eme d'\'equations polynomiales.\/}  
C. R. Acad. Sci. Paris {\bf 345} (2007) 335-340.  
  
\bibitem[PS07b]{PS07}  
P. Philippon, M. Sombra,  
\newblock{\it A refinement of the Ku\v{s}nirenko-Bern\v{s}tein estimate.\/}  
\newblock  Adv. Math.  218  (2008),  ~{\bf 5}, 1370--1418.
  
\bibitem[PV00]{PV00}  
B. Poonen, F. Rodr\'iguez-Villegas,  
{\it Lattice polygons and the number 12.\/}  
Amer. Math. Monthly {\bf 107} (2000) 238-250.  
  
\bibitem[Sch93]{Sch93}  
{R. Schneider,}  
\newblock {\sl  Convex bodies: the Brunn-Minkowski theory.\/}  
\newblock Cambridge Univ. Press, 1993.  
  
\bibitem[SW01]{SW01}  
J.R. Sendra, F. Winkler,  
\newblock{\it Tracing index of rational curve parametrizations.\/}  
\newblock  Comput. Aided Geom. Design  {\bf 18}  (2001) 771-795.  
  
  
\bibitem[Sha94]{Sha94}  
I.R. Shafarevich,  
\newblock {\sl Basic algebraic geometry. 1.  
Varieties in projective space. Second edition.\/}  
Springer, 1994.  
  
\bibitem[Shi74]{Shi74}  
T. Shioda,  
{\it An example of unirational surfaces in characteristic $p$.\/}  
Math. Annalen {\bf 211} (1974) 233-236.  
  
\bibitem[ST07]{ST07}  
B. Sturmfels, J. Tevelev,  
\newblock {\it Elimination theory for tropical varieties.\/}  
\newblock  Math. Res. Lett.  ~{\bf 15}  (2008),  no. 3, 543--562.
  
  
\bibitem[STY07]{STY07}  
B. Sturmfels, J. Tevelev, J.  Yu,  
\newblock {\it The Newton polytope of the implicit equation.\/}  
\newblock Moscow Math. J.  {\bf7}  (2007) 327-346.  
  
\bibitem[SY07]{SY07}  
B. Sturmfels, J.  Yu,  
\newblock{\it Tropical implicitization and mixed fiber polytopes.\/}  
\newblock To appear in  
{\sl Software for Algebraic Geometry},  
IMA Volumes in Mathematics and its Applications, Springer.  
\newblock E-print {\tt arxiv:0706.0564}, 20~pp..  
  
  
\bibitem[SY94]{SY94}  
B. Sturmfels,  J.-T. Yu,  
\newblock{\it Minimal polynomials and sparse resultants.\/}  
\newblock In {\sl Zero-dimensional schemes (Ravello, 1992)}  
317-324, De Gruyter, 1994.  
  
\bibitem[Tev07]{Tev07}  
J. Tevelev,  
\newblock{\it Compactifications of subvarieties of tori.\/}  
Amer. J. Math. {\bf 129} (2007) 1087-1104.  
  
\bibitem[Wal50]{Wal50}  
R.J. Walker,  
\newblock {\sl Algebraic Curves.\/}  
\newblock Princeton Univ. Press, 1950.  
\end{thebibliography}

\end{document}